\documentclass[reqno]{amsart}

\usepackage[title]{appendix}
\usepackage[french,english]{babel}
\usepackage[utf8]{inputenc}
\usepackage[T1]{fontenc}
\usepackage{comment}
\usepackage{palatino}
\usepackage[autostyle]{csquotes}
\usepackage{amsmath}
\usepackage{amsfonts,amssymb,amsmath,latexsym,mathrsfs,bbm,stmaryrd}
\usepackage[all]{xy}
\usepackage[usenames]{color}
\usepackage{epsfig}
\usepackage{graphicx}
\usepackage{subcaption}
\usepackage{float,enumerate}
\usepackage{amsthm}
\usepackage{thmtools}
\usepackage{thm-restate}
\usepackage[mathcal]{euscript}
\usepackage{times}
\usepackage{xcolor}
\usepackage[export]{adjustbox}
\usepackage{hyperref}

\theoremstyle{plain}
\newtheorem{theorem}{Theorem}[section]

\newtheorem{assumption}{Assumption}

\newtheorem{proposition}[theorem]{Proposition}

\newtheorem{lemma}[theorem]{Lemma}
\newtheorem{corollary}[theorem]{Corollary}

\theoremstyle{definition}
\newtheorem{definition}[theorem]{Definition}
\newtheorem{example}[theorem]{Example}
\newtheorem{remark}[theorem]{Remark}

\numberwithin{equation}{section} 

\newtheorem{assumpy}{Assumption}

\newcounter{tmp}
\usepackage[T1]{fontenc}

\newcommand{\un}{1\mkern -4mu{\rm l}}

\def\C{{\mathbb C}}

\def\R{{\mathbb R}}

\def\<{{\langle}}
\def\>{{\rangle}}

\DeclareMathOperator{\supp}{supp}

\begin{document}

\title{Outlier eigenvalues for full rank deformed single ring random matrices}

\author{
Ching-Wei Ho, Zhi Yin and Ping Zhong
}

\address{
\parbox{\linewidth}{Ching-Wei Ho,
Institute of Mathematics, 
Academia Sinica, \\
Taipei 10617, Taiwan. \\
\texttt{chwho@gate.sinica.edu.tw}}
}

\address{
\parbox{\linewidth}{Zhi Yin,
School of Mathematics and Statistics, 
Central South University,\\
Changsha, Hunan 410083, China.\\
\texttt{hustyinzhi@163.com}}
}

\address{
	\parbox{\linewidth}{Ping Zhong,
	Department of Mathematics,
	University of Houston,\\
	Houston, TX 77004, USA.\\
	\texttt{pzhong@central.uh.edu}}
	}

\date{\today}
\maketitle

\begin{abstract}
Let $A_n$ be an $n \times n$ deterministic matrix and $\Sigma_n$ be a deterministic non-negative matrix such that $A_n$ and $\Sigma_n$ converge in $*$-moments to operators $a$ and $\Sigma$ respectively in some $W^*$-probability space. We consider the full rank deformed model $A_n + U_n \Sigma_n V_n,$ where $U_n$ and $V_n$ are independent Haar-distributed random unitary matrices. 
In this paper, we investigate the eigenvalues of $A_n + U_n\Sigma_n V_n$ in two domains that are outside the support of the Brown measure of $a +u \Sigma$. We give a sufficient condition to guarantee that outliers are stable in one domain, and we also prove that there are no outliers in the other domain. When $A_n$ has a bounded rank, the first domain is exactly the one outside the outer boundary of the single ring, and the second domain is the inner disk of the single ring. Our results generalize the work of Benaych-Georges and Rochet (Probab. Theory Relat. Fields, 2016). 
\end{abstract}




\section{Introduction}

\subsection{Deformed single ring model}

For each $n \geq 1,$ let $\Sigma_n = {\rm diag}  (s_1, s_2, \ldots, s_n)$ be an $n \times n$ (deterministic) diagonal matrix, where $s_1, s_2, \ldots, s_n$ are non-negative numbers. Let $U_n$ and $V_n$ be two independent random unitary matrices 
that are Haar-distributed. The single ring random matrix model, introduced in physics literature \cite{FeinbergSZ2001-single-ring-2}, is a bi-unitarily invariant random matrix model $U_n\Sigma_n V_n$ with singular values given by the diagonal entries of $\Sigma_n$.
As a consequence of Voiculescu's asymptotic freeness result \cite{Voiculescu1991}, the random matrices $U_n\Sigma_n V_n$ converge in $*$-moments to an $R$-diagonal operator $T$, a family of non-self-adjoint operators introduced by Nica--Speicher \cite{NicaSpeicher-Rdiag}. Moreover, it is known that the eigenvalue distribution of $U_n\Sigma_n V_n$ converges in probability to a deterministic law as $n$ goes to infinity, due to the fundamental works on this model \cite{GuionnetKZ-single-ring2011, RV2014jams}. This limiting distribution is supported in a single ring centered at the origin and can be identified as the Brown measure of the $R$-diagonal operator $T$ \cite{HaagerupLarsen2000}. 

Recently, there has been a significant progress on the Brown measure of a sum of two free random variables $x+y$ (see \cite{BelinschiYinZhong2021Brown, BercoviciZhong2022,HoHall2020Brown, Ho2020Brown, HoZhong2020Brown, Zhong2021Brown_ctgt, AltK2024Brown}). The Brown measure of $x+y$ offers a suitable candidate for the limiting eigenvalue distribution of the sum of two independent natural random matrix models. In this paper, we study the following deformed random matrix model
\begin{equation}\label{eq:modelY}
A_n + U_n \Sigma_n V_n,
\end{equation}
where $A_n$ is an $n \times n$ matrix (possibly random, but independent of $U_n, V_n, $ and $\Sigma_n$). The {\it empirical spectrum distribution} (ESD) of $X_n\in M_n(\mathbb{C})$ is given by 
\begin{equation*}
  \mu_{X_n}= \frac{1}{n}\left( \delta_{\lambda_1}+\cdots +\delta_{\lambda_n}\right), 
\end{equation*}
where $\lambda_1,\ldots, \lambda_n$ are the eigenvalues of $X_n$. 

Free probability theory and the theory of operator algebras provide a suitable framework to describe the limits of large random matrices. Let $\mathcal{A}$ be a finite von Neumann algebra equipped with a faithful normal tracial state $\tau$. We call the pair $(\mathcal{A},\tau)$ a $W^*$-noncommutative probability space.  We shall assume that $A_n$ converges in $*$-moments to $a\in\mathcal{A}$. That is, 
\[
  A_n\longrightarrow {a} \qquad \text{in $*$-moments}.
\]
Then Voiculescu's asymptotic freeness result implies that $a, T$ are freely independent and 
\[
  A_n+U_n\Sigma_n V_n \longrightarrow {a+T} \qquad \text{in $*$-moments}.
\]
However, the convergence in $*$-moments does not imply convergence in eigenvalue distribution in general. The single ring theorem says that $\mu_{_{U_n\Sigma_n V_n}}$ converges to the Brown measure of $T$ in probability \cite{GuionnetKZ-single-ring2011}. The Brown measure of $T$ was calculated by Haagerup-Larsen  \cite{HaagerupLarsen2000}, and the Brown measure of $a+T$ was calculated by Bercovici and the third author \cite{BercoviciZhong2022}. Built on the Brown measure result of $a+T$ and random matrix techniques developed in \cite{GuionnetKZ-single-ring2011,BenaychGeorges2017,BaoES2019singlering}, the second and the third authors showed that $\mu_{A_n +U_n\Sigma_n V_n}$ converges in probability to the Brown measure of $a+T$ under some mild assumptions, which we will describe in Subsection \ref{subsec:intro-Brown}. 

Suppose now that $A_n=A_n^{'}+A_n^{''}$, where $A_n^{'}$ is well-conditioned and $A_n^{''}$ is a finite rank matrix whose eigenvalues are not from the support of $a+T$. Then the limiting eigenvalue distribution of $A_n+U_n \Sigma_n V_n$ may still be the Brown measure of $a+T$. However, the low rank perturbation $A_n^{''}$ may now create some eigenvalues that are not in the support of the limiting distribution, called outliers. The main objective of this paper is to study the existence and the behavior of the outliers. Our work parallels with Bordenave and Capitaine's work on the full rank deformation of i.i.d. random matrices \cite{BordenaveC_cpam2016_outlier}. 

\subsection{The Brown measure and convergence of eigenvalue distributions}\label{subsec:intro-Brown}

We recall the definition of Brown measure as follows \cite{Brown1986}. The Fuglede-Kadison determinant of $x\in \mathcal{A}$ is defined by
\[
\Delta(x)=\exp \left( \int_0^\infty \log t \; {\rm d}\mu_{\vert x\vert}(t) \right),
\]
where $\mu_{\vert x\vert}$ is the spectral measure of $\vert x\vert$ with respect to $\tau.$  
It is shown that the function ${z}\mapsto\log\Delta(x-z)$ is subharmonic \cite{Brown1986} and there exists  a unique probability measure $\mu_x$ on $\mathbb{C}$ such that
\[
\log\Delta(x-z)=\int_\mathbb{C}\log |\lambda-z| \; {\rm d} \mu_x(\lambda), \; \lambda \in \mathbb{C}. 
\]
The measure $\mu_x$ is the Riesz measure associated with the above subharmonic function and is called the \emph{Brown measure} of $x$.
We note that when $\mathcal{A}=M_n(\C)$ and $\tau= {\rm tr}_n$ is the normalized trace on $M_n(\mathbb{C})$,  for any $X_n \in M_n(\C)$, the Brown measure of $X_n$ is exactly the eigenvalue distribution of $X_n$. The reader is referred to the research monograph \cite{MingoSpeicherBook} for the basics of free probability and Brown measure.

We review some results for the Brown measure $\mu_{a +T}$ obtained by Bercovici and the third author  \cite{BercoviciZhong2022}.
For a probability measure $\mu$ on $\R$, we denote its Cauchy transform by 
\begin{equation}
	\label{eq:CauchyTransDef}
	G_\mu(z) = \int_{\mathbb{R}}\frac{1}{z-t}\; {\rm d}\mu(t),\quad z\in\C^+.
\end{equation}
We also write $\widetilde\mu$ to be the \emph{symmetrization} of $\mu$ defined by
\[\widetilde\mu(B) = \frac{1}{2}[\mu(B)+\mu(-B)]\]
for all Borel set $B\subset \R$. 
Denote
\begin{equation}
	\mu_1 = \tilde{\mu}_{|a -z|} \; \text{and} \; \mu_2 = \tilde{\mu}_{|T|}, \qquad z \in\mathbb{C}. 
\end{equation} 
Let $\omega_1$ and $\omega_2$ be the subordination functions associated with the free convolution $\mu_1 \boxplus \mu_2$ such that
\begin{equation}\label{eq:subordination-function}
	G_{\mu_1  \boxplus \mu_2} (z) =  G_{\mu_1} (\omega_1 (z)) = G_{\mu_2} (\omega_2 (z)), \; z \in \mathbb{C}^+.
\end{equation}
Note that $\mu_1, \omega_1$ and $\omega_2$ depend on the parameter $z$. We may also use $\mu_1$ and  $\mu_2$ to denote generic probability measures when there is no ambiguity. 
It is known that $\omega_1$ and $\omega_2$ can be extended as continuous functions $\omega_1, \omega_2: \overline{\mathbb{C}^+}  \to  \overline{\mathbb{C}^+} \cup \{\infty\}$ (see \cite{Belinschi2008, BelinschiBercoviciHo2022}).

For $p\in\mathbb{R}$, we denote 
\begin{equation*}
	m_p(x)=\tau(|x|^p) \qquad\text{and}\qquad
 m_p(\mu)=\int_\mathbb{R}|t|^p\; {\rm d}\mu(t)
\end{equation*}
for the $p$-th absolute moment of $x\in\mathcal{A}$ or a probability measure $\mu$ on $\mathbb{R}$. If $m_2(\mu)\neq 0$, then the Cauchy-Schwarz inequality implies that 
\begin{equation} 
   \label{eqn:1.5Cauchy}
  m_{-2}(\mu)m_2(\mu)\geq 1,
\end{equation}
and equality holds precisely when $\mu$ is a probability measure supported on two points $\{a, -a\}$ for some $a>0$. If $\mu$ is a symmetric probability measure supported on $\mathbb{R}$, then equality holds when 
\[
  \mu=\frac{1}{2}(\delta_a+\delta_{-a})
\]
for some $a>0$. 
We denote some subsets of $\mathbb{C}$ as follows. 
\begin{equation}\label{eq:sets}
	\begin{split}
		S &= \left\{ z \in \mathbb{C}: \omega_1(0) = \omega_2 (0)=0 \right\},\\
		F_1 &= \left\{ z \in \mathbb{C}: m_2 (|T|) \leq \frac{1}{m_{-2} (|a- z|)}\right\},\\
		F_2 &= \left\{ z \in \mathbb{C}: m_2 (|a-z|) \leq \frac{1}{m_{-2} (|T|)}\right\},\\
		F &= F_1 \cap F_2, \; \text{and}  \\
    \Omega&=\Omega(T,a) = \mathbb{C}\backslash (S \cup F_1 \cup F_2).
	\end{split}
\end{equation}
It is shown that $S$ is a finite set, $F_1$ and $F_2$ are closed sets, and $\Omega$ is an open subset in $\mathbb{C}.$ Moreover, $F_2$ is a closed disk centered at $\tau (a)$ if it is nonempty. The reader is referred to  \cite[Lemma 3.2]{BercoviciZhong2022} for details. The following theorem describes the support of the Brown measure $\mu_{a+T}.$

\begin{theorem}\cite[Theorem 3.6]{BercoviciZhong2022}
	\label{thm:supportOFaPlusT}
	Suppose that $a$ and $T$ are freely independent in $(\mathcal{A}, \tau)$, and $T$ is R-diagonal. Then the support of the Brown measure $\mu_{a+ T}$ is contained in the closure of $\Omega(T,a)$. Moreover, the density formula of $\mu_{a+T}$ can be expressed in terms of the subordination functions $\omega_1$ and $\omega_2$. 
\end{theorem}

We shall consider the following assumptions on the matrices $A_n$ and $\Sigma_n$:

\begingroup
\setcounter{tmp}{\value{assumption}}
\setcounter{assumption}{0} 
\renewcommand\theassumption{(A-\arabic{assumption})}

\begin{assumption}
	\label{assump:Aa}
	The matrix $\Sigma_n$ is invertible and $\|\Sigma_n^{-1}\|\leq n^{\alpha}$ for some $\alpha>0$ independent of $n$. 
\end{assumption}

\begin{assumption}
	 \label{assump:Ab}
	There exists $C_{\ref{assump:Ab}}>0$ independent of $n$ such that
	\begin{equation}
		\Vert A_n \Vert, \Vert \Sigma_n \Vert \leq C_{\ref{assump:Ab}}.
	\end{equation}
\end{assumption}

\begin{assumption}
	 \label{assump:Ac}
  There exists a Lebesgue measurable set $E\subset \Omega(T,a)^c$ that has Lebesgue measure zero with the following property: for any compact set $\Gamma\subset\Omega(T,a)^c\cap E^c$, there exist constants $\kappa_1, \kappa_2>0$ such that 
 \begin{equation}
 	\label{eqn:condition1a}
 	\left\vert G_{\widetilde\mu_{\vert A_n- z\vert}}(i\eta)\right\vert\leq\kappa_2
 \end{equation}
 for all $\eta > n^{-\kappa_1}$ and $ z \in \Gamma$.
\end{assumption}

\begin{assumption}
	 \label{assump:Ad}
	There exist constants $\kappa_1, \kappa_2>0$ such that 
	\begin{equation}
		\label{eqn:condition1b}
		\left\vert G_{\widetilde{\mu}_{_{\Sigma_n}}}(i\eta)\right\vert\leq\kappa_2
	\end{equation}
	for all $\eta > n^{-\kappa_1}$.
\end{assumption}
\endgroup

The first and third author studied the convergence of the ESD of $A_n+U_n\Sigma_n V_n$ to the Brown measure of $a+T$, where $a$ is the limiting operator of $A_n$ in terms of $*$-moments. 
\begin{theorem}\cite{Ho-Zhong23}	\label{thm:maintheorem}
	Suppose that Assumptions \ref{assump:Aa} and \ref{assump:Ab} hold, and either Assumption \ref{assump:Ac} or Assumption \ref{assump:Ad} holds. 
	Then the empirical eigenvalue distribution of $A_n+U_n\Sigma_n V_n$ converges weakly to the Brown measure of $a+T$ in probability
	as $n$ goes to infinity.
\end{theorem}


\subsection{Main results}

Given $X_n \in M_n(\mathbb{C})$, we denote by $s_1(X_n)\geq s_2(X_n)\geq \cdots\geq s_n(X_n)$ the singular values of $X_n$. In particular, $s_1(X_n)=\|X_n\|$ and $s_n(X_n)^{-1}=\| X_n^{-1}\|$ if $X_n$ is invertible.

Denote by ${\rm spec} (x)$ the spectrum of $x$ in $\mathcal{A}$.
For any $z \in \mathbb{C},$ denote by $\mu_z$ the distribution of $(a + T - z)(a + T -z)^*.$ Since $\tau$ is faithful and tracial, we have that (see \cite[Remark 1.4]{Serban2021})
\begin{equation}
{\rm spec} (a+T) = \{z \in \mathbb{C}: 0 \in \supp (\mu_z)\}.
\end{equation}
Note that $z$ is an eigenvalue of $A_n +U_n \Sigma_n V_n$ if and only if $0$ is an eigenvalue of $(A_n+ U_n \Sigma_n V_n-z) (A_n+U_n \Sigma_n V_n- z)^*.$
We will assume the following similar property holds for the operator $a+T$.
\begingroup
\setcounter{tmp}{\value{assumption}}
\setcounter{assumption}{0} 
\renewcommand\theassumption{(X-\arabic{assumption})}
\begin{assumption}\label{assump:Xa}
	$\supp (\mu_{a +T}) = {\rm spec} (a+T)$.
\end{assumption}
\smallskip

We consider the outer domain $\Theta_{\rm out} \subseteq F_1$ and inner domain $\Theta_{\rm in} \subseteq F_2$ as follows:
\begin{equation} 
 \label{eqn:theta.out}
\Theta_{\rm out} := \left\{ z \in \mathbb{C}: m_2 (|T|) < \frac{1}{m_{-2} (|a- z|)} \right\}
\end{equation}
and
\begin{equation}
 \label{eqn:theta.in}
\Theta_{\rm in} := \left\{ z \in \mathbb{C}: m_2 (|a-z|) < \frac{1}{m_{-2} (|T|)} \right\}.
\end{equation}
Note that $\Theta_{\rm in}\cap \Theta_{\rm out}=\emptyset$ due to \eqref{eqn:1.5Cauchy}. 
Throughout this article, the compact subset $\Gamma\subset\mathbb{C}$ is either contained in $\Theta_{\rm out}$ or $\Theta_{\rm in}$. 
\begin{assumption}[$A_n^{'}$ is well-conditioned \cite{BordenaveC_cpam2016_outlier}]
	\label{assump:Xb}
	 $A_n = A^{'}_n+ A^{''}_n,$ where $A^{''}_n$ has a finite rank $r$ independent of $n$ and there exists $C_{\ref{assump:Xb}}>0$ such that for all $n$, $\|A_n^{'}\|+ \|A_n^{''}\| \leq C_{\ref{assump:Xb}}$; and for any $z \in \Gamma,$ there exists $\eta:= \eta_z >0$ such that for all large $n$, 
	\begin{equation}\label{eq:lowerbound-s_n}
		 s_n(A^{'}_n- z)\geq \eta. 
	\end{equation}
	\end{assumption}
Note that if Assumption \ref{assump:Xb} holds, then $A_n^{'}-z$ is invertible and there exists $\eta':=\eta_z' >0$ such that for all large $n,$
		\begin{equation}
			s_n((A_n^{'} -z)^{-1} )\geq \eta'.
		\end{equation}

\begin{assumption}\label{assump:Xd}
	$\Sigma_n$ is invertible for all $n$, and there exists $C_{\ref{assump:Xd}}>0$ independent of $n$ such that
		\begin{equation*}
			\Vert \Sigma_n^{-1} \Vert \leq C_{\ref{assump:Xd}}.
		\end{equation*}
\end{assumption}
\endgroup

\bigskip

We are ready to state our main results. Our first result gives a sufficient condition to guarantee that outliers are stable when they are in the domain $\Theta_{\rm out}$. More precisely, we show that the number of outliers of $A_n$ and of $A_n+U_n\Sigma_n V_n$ in $\Theta_{\rm out}$ are equal and their locations are asymptotically close to each other.

\begin{theorem}\label{thm:outlier}
	Assume that Assumption \ref{assump:Ab} holds and Assumptions \ref{assump:Xa}--\ref{assump:Xb} hold for  a compact set  $\Gamma \subseteq \Theta_{\rm out}$ with continuous boundary. 
	If for some $\varepsilon>0$ and all $n$ large enough,
	\begin{equation}\label{eq:condition-stable-outlier}
		\min_{z \in \partial \Gamma} \left\vert  \frac{\det (A_n-z)}{\det(A^{'}_n-z)} \right\vert \geq \varepsilon,
	\end{equation}
	then in probability for all large $n$, the number of eigenvalues of $A_n$ and of $A_n+U_n\Sigma_n V_n$ in $\Gamma$ are equal.
\end{theorem}

Our second main result claims that there is no outlier in the domain $\Theta_{\rm in}$. 

\begin{theorem}\label{thm:outlier-inner}
Assume that Assumption \ref{assump:Ab} holds and Assumptions \ref{assump:Xa}-\ref{assump:Xb}-\ref{assump:Xd} hold for  a compact set  $\Gamma \subseteq \Theta_{\rm in}$ with continuous boundary. 
Then, in probability for all large $n$, $A_n+U_n\Sigma_n V_n$ has no eigenvalue in $\Gamma.$
\end{theorem}

\begin{remark}
It is clear that Assumption \ref{assump:Xd} implies Assumption \ref{assump:Aa}. Also note that we do not need any assumption (Assumption \ref{assump:Ac} or \ref{assump:Ad}) on the convergence of the empirical spectral measure of $A_n+U_n\Sigma_n V_n.$
According to the assertion of Proposition \ref{prop:no-outlier-2}, for any $z \in \mathbb{C},$ almost surely, the empirical spectral measure of $(A_n+U_n\Sigma_n V_n-z)(A_n+U_n\Sigma_n V_n-z)^*$ converges weakly to $\mu_z.$ 
\end{remark}

\begin{figure}
\includegraphics[scale=0.8, cframe=black!50!black 0.2mm]{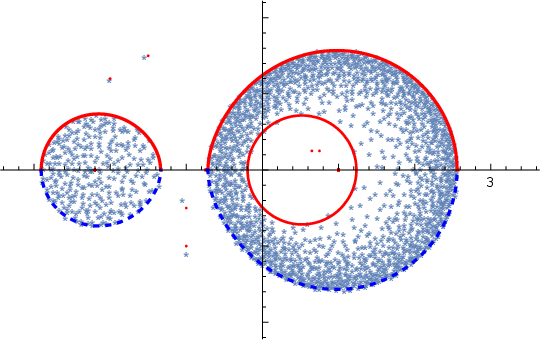}\caption{\label{fig:outlier}The matrix $A_n^{''}$ has eigenvalues (marked red) at $0.75 + 0.25i$, $0.65+0.25i$, $-1.5+1.5i$, $-2 + 1.2 i$, $-1 - i,$ and $-1-0.5i$. The outliers of $A_n + U_n\Sigma_n V_n$ are near $-1.5+1.5i$, $-2 + 1.2 i$, $-1 - i$ and $-1-0.5i$.}
\end{figure}

\begin{example}
We choose $\Sigma_n$ such that $\mu_{\Sigma_n}$ converges to $\mu_{\Sigma} =0.4*\delta_{1}+0.6*\delta_2,$ and $A_n^{'}$ is a diagonal matrix such that $\mu_{A^{'}_n}$ converges to 
$0.15*\delta_{-2.2}+0.85*\delta_1$. Let $A^{''}_n$ be a diagonal matrix such that $A^{''}_n (1,1)=0.75+0.25i$, $A^{''}_n(2,2)=0.65+0.25i$, $A^{''}_n(3,3)=-1.5+1.5i$, $A^{''}_n(4,4)=-2+1.2i,$ $A^{''}_n(5,5)=-1-i$ and $A^{''}_n(6,6)=-1-0.5i$ and all other entries of $A^{''}_n$ are zero. The random matrix $A^{'}_n+A^{''}_n+ U_n\Sigma_n V_n$ has outliers near $-1.5+1.5i$, $-2 + 1.2 i$, $-1 - i$ and $-1-0.5i$ as shown in Figure \ref{fig:outlier},
\end{example}


There is an extensive list of works on outliers of Hermitian random matrix models (see \cite{KnowlesYin2024aop, PRSoshnikov2013AIHP, CDFF2011, BBCF2017aop, BBC2021imrn} and references therein). In \cite{Tao-2013}, Tao initiated the study of outliers of finite rank perturbations of the i.i.d. random matrix model. The result was then extended to full rank perturbations of the i.i.d. random matrix model by Bordenave-Capitaine \cite{BordenaveC_cpam2016_outlier}. 
By assuming some Brown measure results, the outliers of a polynomial of i.i.d. random matrix models were studied in \cite{CDFF2011}. The single ring random matrix model is another well-studied non-Hermitian random matrix model. The outliers of finite rank perturbations of the single ring model were studied by Benaych-Georges and Rochet \cite{Benaych-Georges-2016}. Our main results extend their results to full rank perturbations of the single ring model. 

It is known that one can use subordination functions in free probability to predict the location of outliers in random matrices \cite{CapitaineD2017survey, BBCF2017aop, BBC2021imrn}. 
The recent work \cite{BercoviciZhong2022} on the Brown measure of $a+T$ provides relevant tools of subordination functions. However, the subordination functions appearing in the study of $a+T$ and the full rank perturbation of $A_n+U_n\Sigma_n V_n$ are not as tractable as those for the study of the i.i.d. random matrix model, as they do not have global left inverses. In Section \ref{sec:free-prob}, we prove that the subordination functions have an analytic continuation in some neighborhood of the origin, which allows us to study the outliers using techniques from \cite{BordenaveC_cpam2016_outlier, Benaych-Georges-2016} in our work. See Proposition \ref{prop:no-atom-omega_1}, Theorem \ref{thm:support}, and Theorem \ref{thm:support-in} for details.


\subsection{Outline of the proofs}

The methods used in this article are adapted from previous works  \cite{Tao-2013, BordenaveC_cpam2016_outlier} on the outliers of the deformed i.i.d. model, and \cite{Benaych-Georges-2016} on the outliers of the finite rank deformed single ring theorem. The skeleton of our proofs is similar to the work of Bordenave-Capitaine \cite{BordenaveC_cpam2016_outlier}. In addition, we utilize some techniques for Brown measure from \cite{BercoviciZhong2022}.

{\bf Model reduction:} We note that matrices $A_n + U_n\Sigma_n V_n$ and $V_n (A_n + U_n\Sigma_n V_n) V_n^*$ have the same spectrum. The product $V_nU_n$ is also Haar-distributed and independent from $V_n$. Hence, instead of studying $A_n + U_n\Sigma_n V_n$, we may consider the following deformed model
\begin{equation}
	M_n=A_n + Y_n,
\end{equation}
where 
\begin{equation}
	Y_n=U_n\Sigma_n,
\end{equation}
and $A_n$ is independent from $Y_n$ and the distribution of $A_n$ is invariant under the conjugation by any unitary matrix.

We will first prove that there is no outlier when $A_n^{''}=0.$
\begin{theorem}\label{thm:no-outlier}
We consider the following two cases. 
\begin{enumerate}[\rm (a)]
\item Let $\Gamma \subseteq \Theta_{\rm out}$ be a compact set with a continuous boundary. Assume that the same assumptions hold as Theorem \ref{thm:outlier} for $\Gamma$ with $A^{''}_n=0$.
\item Let $\Gamma \subseteq \Theta_{\rm in}$ be a compact set with a continuous boundary. Assume that the same assumptions hold as Theorem \ref{thm:outlier-inner} for $\Gamma$ with $A^{''}_n=0.$
\end{enumerate}
Then in both cases, almost surely for all large $n$, $M_n$ has no eigenvalue in $\Gamma.$
\end{theorem}

Recall that $s_n(X_n)$ denotes the smallest singular value of $X_n$. Note that $z$ is an eigenvalue of $X_n$ if and only if $0$ is a singular value of $X_n -z.$
Hence, Theorem \ref{thm:no-outlier} is equivalent to the following proposition. 

\begin{proposition}
Let $\Gamma$ be a compact set with a continuous boundary that satisfies the assumptions of cases $(a)$ or $(b)$ of Theorem \ref{thm:no-outlier}. 
Then, for any $z \in \Gamma$, there exists $\gamma_z>0,$ such that almost surely, for all large $n$, $s_n(M_n - z)\geq \gamma_z$. Consequently, there exists 
$\gamma_\Gamma$ such that almost surely, for all large $n$, $\inf_{z \in \Gamma} s_n (M_n - z) \geq \gamma_{\Gamma}.$
\end{proposition}

The proof of the above proposition is divided into two parts following the approach in \cite{Serban2021, BordenaveC_cpam2016_outlier}.  The first part is devoted to studying the spectrum of  
$a_n +T-z,$ where $\{a_n\}$ is a sequence of operators in $\mathcal{A}$ such that $a_n$ is free from $T$ and converges in distribution to $a$ (see Proposition \ref{prop:no-outlier-1} for $z \in \Theta_{\rm out}$ and Corollary \ref{cor:no-outlier-1} for $z \in \Theta_{\rm in}$). Our proof uses the subordination functions in free probability theory. The second part is to compare the spectrum of $(M_n-z)(M_n -z)^*$ with the support of the spectral measure of $(a+T-z)(a+T-z)^*$ (see Proposition \ref{prop:no-outlier-2}). To this end, we will use Collins-Male's approach \cite{CM2014}, which enables us to transfer our model to the polynomial of GUEs and deterministic matrices. Finally, we apply the strong convergence result from the work of Belinschi and Capitaine \cite{Serban2017JFA}. 

The method for proving Theorem \ref{thm:outlier} and Theorem \ref{thm:outlier-inner} is now quite standard (see \cite{Tao-2013} and \cite{BordenaveC_cpam2016_outlier} for example). Let $P_n \in M_{n \times r} (\mathbb{C})$ and $Q_n \in M_{r \times n} (\mathbb{C})$ such that $A_n^{''} = P_n  Q_n.$ For any $z \in \mathbb{C},$ we introduce the resolvent matrices
\begin{equation}
R_n (z) = (A_n^{'} + Y_n -z)^{-1} \;\; \text{and} \; R_n^{'}(z) = (A_n^{'} -z)^{-1}.
\end{equation}
Then we have 
\begin{equation}
 \label{eqn:1.19identityDET1}
\begin{split}
 \det \left( A_n +Y_n-z  \right) &= \det \left( A_n^{'}+ Y_n -z\right)\cdot  \det \left( \un_n - R_n(z) A_n^{''} \right)\\
 &= \det \left( A_n^{'}+ Y_n -z\right)\cdot  \det \left( \un_r - Q_n R_n(z) P_n \right),
\end{split}
\end{equation}
where $\un_n$ is the identity matrix in $M_n(\mathbb{C})$. 
Define an analytic function
\begin{equation}
f_n(z) = \det \left( \un_r - Q_n R_n(z) P_n \right).
\end{equation}
Hence, the eigenvalues of $A_n + Y_n$ which are not eigenvalues of $A_n^{'} + Y_n$ must be zeros of the function $f_n.$ 

Define 
\begin{equation}
g_n(z) = \det \left( \un_r - Q_n R_n^{'}(z) P_n \right).
\end{equation}
By a similar calculation as \eqref{eqn:1.19identityDET1}, we obtain
\begin{equation*}
g_n(z) = \frac{\det \left( A_n-z \right)}{\det \left( A_n^{'}-z \right)}.
\end{equation*}
Then, thanks to Rouch\'{e}'s theorem, the key ingredient for proving Theorem \ref{thm:outlier} is to show that $f_n(z) \approx g_n(z)$ on $\Gamma \subseteq \Theta_{\rm out}.$ To this end, we will consider the difference $Q_n R_n(z) P_n - Q_n R_n^{'}(z) P_n,$ and show that
\begin{equation}\label{eq:differ.convergence}
v^* (R_n^{'}(z)Y_n)^k R_n^{'}(z) u \to 0
\end{equation}
for any unit column vectors $v$ and $u.$ It follows that
\begin{equation}\label{eq:differ}
Q_n R_n(z) P_n -Q_n R_n^{'}(z) P_n = \sum_{k\geq 1}(-1)^k Q_n (R_n^{'}(z) Y_n)^k R_n^{'}(z) P_n.
\end{equation}
Finally, our result holds by proving that $\Vert Q_n R_n(z) P_n - Q_n R_n^{'}(z) P_n \Vert \to 0$ in probability as $n$ goes to infinity (see Proposition \ref{prop:norm-difference}). 

For proving Theorem \ref{thm:outlier-inner}, we shall consider the zeros of $f_n$ for $z \in \Theta_{\rm in}.$
A simple condition for $f_n (z) =0$ would be 
\begin{equation}
\Vert Q_n R_n(z) P_n \Vert <1.
\end{equation}
We first show that the spectrum radius of 
$(A_n^{'}-z) Y_n^{-1}$ is strictly less than $1$ (see Proposition \ref{prop:spectral-radius-1}) for any $z \in \Theta_{\rm in}.$ Then we have 
\begin{equation*}
\begin{split}
Q_n R_n(z) P_n &  =Q_n Y_n^{-1} \left( \un_n- (z- A_n^{'}) Y_n^{-1} \right)^{-1} P_n\\
& = Q_n \sum_{k=0}^\infty Y_n^{-1}  ((z- A_n^{'}) Y_n^{-1})^k P_n.
\end{split}
\end{equation*}
We will deduce our result from estimates for the proof of \eqref{eq:differ.convergence} for $A_n-z$ and $Y_n$ replaced by $(A_n -z)^{-1}$ and $Y_n^{-1}$ respectively (see Proposition \ref{prop:norm-convergence-inner}). 

The rest of the paper is organized as follows. We develop some auxiliary results in free probability and Brown measures that are essential to studying the full-rank perturbed single ring matrix model. Section \ref{sec:no-outlier} is devoted to the proof of Theorem \ref{thm:no-outlier}, and Section \ref{sec:stable-outlier} is for Theorem \ref{thm:outlier}. In the last section, we finish the proof of Theorem \ref{thm:outlier-inner}.


\section{Auxiliary results in free probability theory}\label{sec:free-prob}

\subsection{Basic notions of free probability}\label{subsec:free-prob}
We recall some basic notions of Voiculescu's free probability theory and refer to \cite{SpeicherNicaBook, MingoSpeicherBook} for more details. 
Let $(\mathcal{A}, \tau)$ be a $W^*$-probability space, where $\mathcal{A}$ is a finite von Neumann algebra and $\tau$ is a faithful normal tracial state on $\mathcal{A}.$ An element of $\mathcal{A}$ is called a noncommutative random variable. A family of elements $(x_i)_{i \in I}$ in $(\mathcal{A}, \tau)$ is freely independent if for $k \in \mathbb{N}$
and polynomials $p_1, \ldots, p_k$ in two noncommutative indeterminates, one has
\begin{equation*}
\tau \left( p_1(x_{i_1}, x^*_{i_1})\cdots p_k (x_{i_k}, x^*_{i_k}) \right) = 0
\end{equation*}
whenever $i_1 \neq i_2 \neq \cdots \neq i_k$ and $\tau (p_l(x_{i_l}, x^*_{i_l})) = 0$ for $l = 1, \ldots, k.$

\begin{definition}\label{def:strong-limit}
Let ${\bf x} = (x_1, \ldots,  x_k)$ be a sequence of elements in $(\mathcal{A}, \tau)$. The joint distribution of ${\bf x}$ is
the linear form $P \to \tau (P ({\bf x}, {\bf x}^*))$ on the set of polynomials in $2k$ noncommutative indeterminates.

Let ${\bf x}_n = (x_1^{(n)}, \ldots,  x_k^{(n)})$ be a $k$-tuple of elements in $(\mathcal{A}_n, \tau_n)$ for each $n \in \mathbb{N}$. By convergence in distribution, we mean the pointwise convergence of the map 
\begin{equation*}
P \to  \tau_n (P ({\bf x}_n, {\bf x}^*_n)),
\end{equation*}
and by strong convergence in distribution, we mean convergence in distribution, and pointwise convergence of the map
\begin{equation*}
P \to  \Vert P ({\bf x}_n, {\bf x}^*_n) \Vert.
\end{equation*}
Moreover, ${\bf x}$ is called the limit (resp. strong limit) of ${\bf x}_n$ if 
$\tau_n (P ({\bf x}_n, {\bf x}^*_n)) \to \tau (P ({\bf x}, {\bf x}^*))$
(resp. $\tau_n (P ({\bf x}_n, {\bf x}^*_n)) \to \tau (P ({\bf x}, {\bf x}^*))$ and $\Vert P ({\bf x}_n, {\bf x}^*_n) \Vert \to \Vert P ({\bf x}, {\bf x}^*) \Vert)$ for any $P.$
\end{definition}

\begin{remark}
Given a self-adjoint element $x$ in $(\mathcal{A}, \tau)$, let $\mu_x$ be the spectral measure of $x$ with respect to $\tau$ such that
\begin{equation*}
\int_{\mathbb{R}} t^k \; {\rm d} \mu_x (t) = \tau (x^k)
\end{equation*}
for any $k \geq 1.$ Moreover, let $(x_n)_{n \geq 1}$ be a sequence of self-adjoint elements such that $ x_n \in (\mathcal{A}_n, \tau_n).$ 
Then, the following statements are equivalent \cite{CM2014}.
\begin{enumerate}[{\rm (1)}]
\item $x_n$ converges strongly in distribution to $x.$
\item $x_n$ converges in distribution to $x.$ Moreover, for any $\varepsilon > 0$, there exists $n_0$ such that for any $n > n_0$,
\begin{equation*}
\supp (\mu_{x_n}) \subseteq \supp(\mu_x) + [-\varepsilon, \varepsilon].
\end{equation*}
\end{enumerate}
\end{remark}


Let $x_1$ and $x_2$ be freely independent self-adjoint elements in $(\mathcal{A}, \tau)$ such that $x_1, x_2$ have distributions $\mu_1, \mu_2$ respectively. Then the distribution of the sum $x_1 + x_2$ is called the free convolution of $\mu_1$ and $\mu_2$ and is denoted by $\mu_1 \boxplus \mu_2.$

\begin{theorem} (\cite[Chapter 3]{MingoSpeicherBook} and \cite{Belinschi2008})
Suppose that $x_1, x_2 \in (\mathcal{A}, \tau)$ are self-adjoint that are freely independent. If neither $x_1$ nor $x_2$ is a constant multiple of the identity operator, then there exist unique continuous functions $\omega_1, \omega_2: \overline{\mathbb{C}^+}  \to  \overline{\mathbb{C}^+} \cup \{\infty\}$ that are analytic on $\mathbb{C}^+$ such that
\begin{equation}\label{eq:subordination}
G_{\mu_1 \boxplus \mu_2} (z) =  G_{\mu_1} (\omega_1(z)) = G_{\mu_2} (\omega_2(z)), \; z \in \mathbb{C}^+.
\end{equation}
\end{theorem}

For any measure $\mu$ on $\mathbb{R},$ define 
\begin{equation*}
F_{\mu} (z) = \frac{1}{G_{\mu}(z)},
\end{equation*}
and the Voiculescu transform is given by
\begin{equation*}
\phi_{\mu} (z) = F_\mu^{-1} (z)-z,
\end{equation*}
where the inverse function and $\phi_\mu$ are defined in a neighborhood of infinity. 
For simplicity, we denote
$G_i (z): =G_{\mu_i} (z), F_i (z):= F_{\mu_i} (z),$ and $\phi_i (z) : = \phi_{\mu_i}(z)$  for $i=1,2.$ It is well-known that
\begin{enumerate}[{\rm (1)}]
\item $\lim_{y \uparrow +\infty}  \frac{\omega_1 (i y)}{i y}=\lim_{y \uparrow +\infty}  \frac{\omega_2 (i y)}{i y} =1$. 
\item $\omega_1(z) + \omega_2(z) = z + F_{\mu_1 \boxplus \mu_2} (z).$
\end{enumerate}
By Nevanlinna representation of holomorphic maps from $\mathbb{C}^+$ to itself, we have 
\begin{equation}\label{eq:Nevanlinna}
\begin{split}
\omega_i (z) &= r_i + z + \int_{\mathbb{R}} \frac{1}{t-z} \; {\rm d} \rho_i (t),\\
F_{\mu_1 \boxplus \mu_2}(z)& = r + z + \int_{\mathbb{R}} \frac{1}{t-z} \; {\rm d} \rho (t),
\end{split}
\end{equation}
where $r_i, r\in \mathbb{R}$ and $\rho_i, \rho$ are positive, finite measures which are uniquely determined by $\omega_i$ and $F_{\mu_1 \boxplus \mu_2}.$ The above property $(2)$ implies that
$r_1 + r_2 = r$ and $\rho_1 + \rho_2 = \rho.$


\subsection{Support of the Brown measure of $a+T$}\label{subsec:sum-R-diagonal}

Given $T \in (\mathcal{A}, \tau)$, the operator $T$ is a R-diagonal element if $T$ has the same $*$-moments as $uh$, where $u$ is a Haar unitary, $h$ is a positive operator, and $u, h$ are freely independent. 
Throughout this article, $a\in\mathcal{A}$ is freely independent from $T$. We denote
\begin{equation}\label{eq:mu-z}
\mu_{1} = \tilde{\mu}_{|a -z|} \; \text{and} \; \mu_2 = \tilde{\mu}_{|T|}.
\end{equation} 
For every $z \in \mathbb{C},$ the self-adjoint elements $|a+ T- z|$ and $| u(a-z) + T |$ have the same distribution. Now $u (a-z)$ is R-diagonal, then we have \cite{HaagerupLarsen2000, HaagerupSchultz2007, NicaSpeicher-Rdiag}
\begin{equation}
\tilde{\mu}_{|a+T -z|} = \mu_{1} \boxplus \mu_2.
\end{equation}
Let $\omega_1$ and $\omega_2$ be the subordination functions associated with the free convolution $\mu_{1} \boxplus \mu_2$. For any probability measure $\mu$ on $\mathbb{R},$ recall that
\begin{equation}
m_2 (\mu) = \int_{\mathbb{R}} t^2 \; {\rm d}\mu (t) \; \text{and} \; m_{-2} (\mu) = \int_{\mathbb{R}} t^{-2}\; {\rm d}\mu (t).
\end{equation}
\begin{proposition}\cite[Proposition 2.9 and Remark 2.11]{BercoviciZhong2022}
\label{prop:m2_mNegative2}
Let $\mu_1, \mu_2$ be symmetric probability measures on $\mathbb{R}$ and $\mu=\mu_1\boxplus \mu_2$. If $\omega_1(0)=0$ and $\omega_2(0)=\infty$, we have 
\[
   m_2(\mu_2)\leq \frac{1}{m_{-2}(\mu_1)},
\]
and 
\begin{equation}
  \label{eqn:limit_yOmega1}
    \lim_{y\downarrow {0}}\frac{\omega_1(iy)}{iy}=\frac{1}{m_{-2}(\mu_1)m_2(\mu_2)},
\end{equation}
and
\begin{equation}
  \label{eqn:limit_yOmega2}
    \lim_{y\downarrow {0}}(-iy\omega_2(iy))=\frac{1}{m_{-2}(\mu_1)}-m_2(\mu_2).
\end{equation}
\end{proposition}
Recall that the sets $S, F_1, F_2,$ and $\Omega$ given in \eqref{eq:sets} are crucial to studying the Brown measure of $a+T.$ And recall from Theorem \ref{thm:supportOFaPlusT} that the set $\overline\Omega$ contains the support of the Brown measure. 
In the sequel, we will use the following descriptions of the sets $F_1$ and $F_2$ from \cite{BercoviciZhong2022}. 

\begin{lemma}\cite[Lemma 3.2]{BercoviciZhong2022}\label{lem:descr-subordination}
Suppose that $a$ and $T$ are freely independent, and $T$ is R-diagonal. Then we have 
\begin{enumerate}[{\rm (1)}]
\item $F_1 = \{z \in \mathbb{C}: \omega_1 (0)=0, \omega_2 (0)= \infty\}.$ In particular, $F_1= \emptyset$ if $m_2 (|T|) = +\infty$. If $F_1\neq \emptyset$ then it is a closed
set.
\item $F_2 = \{z \in \mathbb{C}: \omega_1 (0)=\infty, \omega_2 (0)= 0\}.$ In particular, $F_2= \emptyset$ if $m_2 (|a|) = +\infty$. In general, $F_2$ is $\emptyset$, a single point 
$\{\tau(a)\}$, or a closed disk centered at $\tau(a)$ with radius $1/m_{-2} (|T|)- m_2 (|a- \tau(a)|).$
\end{enumerate}
\end{lemma}

In this paper, we will study the eigenvalues of $M_n$ in the subsets $\Theta_{\rm out}$ and $\Theta_{\rm in}$ of $F_1$ and $F_2$.


\subsection{Properties of the support via subordination functions}\label{subsec:free-convolution}

Recall that the outer domain $\Theta_{\rm out}$ and the inner domain $\Theta_{\rm in}$ are given in \eqref{eqn:theta.out} and \eqref{eqn:theta.in} as follows:
\begin{equation*}
\Theta_{\rm out} := \left\{ z \in \mathbb{C}: m_2 (\mu_2) < \frac{1}{m_{-2} (\mu_1)} \right\} 
\end{equation*}
and
\begin{equation*}
\Theta_{\rm in} := \left\{ z \in \mathbb{C}: m_2 (\mu_1) < \frac{1}{m_{-2} (\mu_2)} \right\}.
\end{equation*}
If $z \in \Theta_{\rm out}$ or $z \in \Theta_{\rm in}$, then Theorem \ref{thm:supportOFaPlusT} concerning the 
support of $\mu_{a+T}$ implies that $0 \notin \supp (\mu_1 \boxplus \mu_2)$ if 
Assumption \ref{assump:Xa} holds. 

\begin{lemma}\label{lem:support}
Let $\mu_1, \mu_2$ be symmetric probability measures on $\mathbb{R}$ with compact support given by \eqref{eq:mu-z}, and denote $\mu=\mu_1\boxplus \mu_2$. Suppose that $0\notin \supp(\mu)$. If $z \in \Theta_{\rm out}$, then $0\notin \supp(\rho_1)$ and $0$ is an isolated point on both $\supp(\rho)$ and $\supp(\rho_2)$. 

Moreover, we have 
\[
  \rho(\{0\})=\rho_2(\{0\})=\frac{1}{m_{-2}(\mu_1)}-m_2(\mu_2)>0.
\]
\end{lemma}

\begin{proof}
Lemma \ref{lem:descr-subordination} implies that $\omega_1 (0) =0$ and $\omega_2(0) = \infty$ for $z \in \Theta_{\rm out}.$
It is clear that $\rho, \rho_1, \rho_2$ are also symmetric measures. Hence $r=r_1=r_2=0$. We have 
\[
  \frac{\omega_1(iy)}{iy}=1+\int_\mathbb{R}\frac{1}{t^2+y^2} \; {\rm d}\rho_1(t).
\]
Note that \eqref{eqn:limit_yOmega1} implies that the limit must be finite, so we obtain $\rho_1(\{0\})=0$. We also have 
\[
    -iy\omega_2(iy)=y^2+\int_\mathbb{R}\frac{y^2}{t^2+y^2} \; {\rm d} \rho_2(t).
\]
By the dominated convergence theorem and \eqref{eqn:limit_yOmega2}, we see that 
\[
\rho_2(\{0\})=\lim_{y\downarrow {0}}(-iy\omega_2(iy))=\frac{1}{m_{-2}(\mu_1)}-m_2(\mu_2).
\]
Since $\rho=\rho_1+\rho_2$, we have $\rho(\{0\})=\rho_2(\{0\})$.

Note that $0$ is the only possible atom of $\mu$ and the assumption excludes this possibility. Hence $\mu$ is absolutely continuous with respect to the Lebesgue measure since $\mu=\mu_1\boxplus\mu_2$ has no singular continuous part \cite{Belinschi2008}. 
Since $0 \notin \supp (\mu),$ there exists $\delta >0$ such that $[-\delta , \delta] \cap \supp (\mu) = \emptyset.$ So the Cauchy transform $G_{\mu}(z)$ can be written as
\begin{equation}
G_{\mu}(z) = \int_{\mathbb{R} \backslash [-\delta, \delta]} \frac{1}{z- t} \; {\rm d} \mu (t). 
\end{equation}
Recall that for any $z \in \mathbb{C}^+$
\begin{equation}
\begin{split}
F_{\mu}(z) =  z + \int_{\mathbb{R}} \frac{1}{t- z} \; {\rm d} \rho (t) =  z - G_\rho (z).
\end{split}
\end{equation}
Then for any $x \in (-\delta, 0) \cup (0, \delta)$ and $y>0,$ by \eqref{eq:Nevanlinna} we have 
\begin{equation*}
\begin{split}
{\rm Im}G_\rho (x + iy) & = y - {\rm Im} F_{\mu}(x+ iy) = y +  \frac{ {\rm Im} G_\mu (x+ iy)}{|G_\mu (x+ i y)|^2}.
\end{split}
\end{equation*}

Note that
\begin{equation*}
{\rm Im} G_\mu (x+ iy) = - y \int_{\mathbb{R} \backslash (-\delta, \delta)}  \frac{1}{(x-t)^2 + y^2} \; {\rm d} \mu (t).
\end{equation*}
Then for any $x \in (-\delta, 0) \cup (0, \delta),$ by the dominated convergence theorem we have 
\begin{equation}\label{eq:IM=0}
\lim_{y \downarrow 0}{\rm Im}G_\mu (x + iy)  = 0.
\end{equation}
On the other hand, for all $x \in (-\delta, \delta),$
\begin{equation*}
\begin{split}
\lim_{y \downarrow 0}| G_\mu (x + iy)|^2 &= \lim_{y \downarrow 0}\left( {\rm Re} G_\mu (x+ i y)\right)^2 \\
& = \left(  \int_{\mathbb{R} \backslash [-\delta, \delta]}  \frac{1}{x-t} \; {\rm d} \mu (t) \right)^2.
\end{split}
\end{equation*}
Since $\mu$ is symmetric, for any $x \in (0, \delta)$
\begin{equation*}
 \int_{\mathbb{R} \backslash [-\delta, \delta]}  \frac{1}{x-t} \; {\rm d} \mu (t) =  \int_{(\delta, +\infty)}  \frac{2 x}{x^2-t^2} {\rm d} \mu (t)<0.
\end{equation*}
Similarly, we have $\int_{\mathbb{R} \backslash [-\delta, \delta]}  \frac{1}{x-t} \; {\rm d} \mu (t) >0$ for any $x \in (-\delta, 0).$
It follows that 
\begin{equation}\label{eq:|G|>0}
\lim_{y \downarrow 0}| G_\mu (x + iy)|^2 >0
\end{equation}
for any $x \in (-\delta, 0) \cup (0, \delta).$ Therefore, combining \eqref{eq:IM=0} and \eqref{eq:|G|>0} we obtain 
\begin{equation}
\lim_{y \downarrow 0}{\rm Im}G_\rho (x + iy)  = 0
\end{equation}
for any $x \in (-\delta, 0) \cup (0, \delta).$ Hence, it follows from the Stieltjes inverse formula that $\rho ((-\delta, 0) \cup (0, \delta)) = 0$, which shows that
$0$ is an isolated point on $\supp(\rho).$

Since $\rho = \rho_1 + \rho_2,$ $\rho_1 ((-\delta, 0) \cup (0, \delta)) = \rho_2 ((-\delta, 0) \cup (0, \delta)) = 0.$ Thus, $0$ is also an isolated point on $\supp(\rho_2).$ Recall that $\rho_1 (\{0\}) =0.$ Then we have $\rho_1 ((-\delta, \delta)) = 0.$ It implies that $0 \notin \supp (\rho_1).$   

It remains to show that $\rho(\{0\})>0$. Suppose $\rho(\{0\})=\rho_2(\{0\})=0$, then $0\notin\supp(\rho_2)$. It implies that $\omega_2(0)=0$ which contradicts to the fact that $\omega_2(0)=\infty$. Therefore, 
$\rho(\{0\})=\rho_2(\{0\})=\frac{1}{m_{-2}(\mu_1)}-m_2(\mu_2)>0$.
\end{proof}

\begin{proposition}\label{prop:no-atom-omega_1}
Let $\mu_1, \mu_2$ be symmetric probability measures on $\mathbb{R}$ with compact support given by \eqref{eq:mu-z}, and denote $\mu=\mu_1\boxplus \mu_2$. Suppose that $0\notin \supp(\mu)$. If $z \in \Theta_{\rm out}$, then there exists $\delta >0$, such that for any $x \in (-\delta, \delta)$, we have $\omega_1(x)\in\mathbb{R}$ and $\omega_1'(x)>1$. Moreover, 
$(-\omega_1(\delta),\omega_1(\delta))\cap \supp(\mu_1)=\emptyset$.
\end{proposition}

\begin{proof}
  By Lemma \ref{lem:support}, there exists $\delta >0$ such that $[-\delta, \delta] \cap \supp (\rho_1) = \emptyset.$ By the proof of Lemma \ref{lem:support}, we can rewrite $\omega_1$ as 
	\[
	    \omega_1(z)=z+\int_{\mathbb{R}\backslash [-\delta, \delta]} \frac{1}{t-z} \; {\rm d} \rho_1(t)
	\]
	since $[-\delta,\delta]\cap \supp(\rho_1)=\emptyset$. 
	Hence, for any $x\in (-\delta,\delta)$, we have 
	\[
	  \omega_1(x)=x+\int_{\mathbb{R}\backslash [-\delta, \delta]} \frac{1}{t-x}\; {\rm d} \rho_1(t).
	\]
	By a direct calculation, $\omega_1(x)\in\mathbb{R}$ and $\omega_1'(x)>1$ for $x\in(-\delta,\delta)$. Note that $\rho_1$ is also a symmetric measure and hence $\omega_1(-\delta)=-\omega_1(\delta)$. Therefore, 
	\[
	     \omega_1\big( (-\delta,\delta) \big)=(-\omega_1(\delta),\omega_1(\delta)).
	\]
	
	The proof of Lemma \ref{lem:support} also shows that 
	\[
	   \lim_{y\downarrow {0}}\text{Im} G_\mu(x+iy)=0
	\]
for any  $x\in(-\delta,\delta).$
	Since $\omega_1$ is analytic in a neighborhood of $(-\delta,\delta)$ and $\omega_1'(x)>1$ for any $x\in(-\delta,\delta)$, it follows that 
	\[
	     \lim_{y\rightarrow {0}}\omega_1(x+iy)=\omega_1(x).
	\]
	Hence, for any $x\in(-\delta,\delta)$, we have
	\begin{align*}
		 \lim_{y\downarrow {0}}\text{Im}G_{\mu_1}(\omega_1(x)+iy)&=
		     \lim_{y\downarrow {0}}\text{Im}G_{\mu_1}(\omega_1(x+iy))\\
		     &=\lim_{y\downarrow {0}}\text{Im}G_\mu(x+iy)
		     =\text{Im}G_\mu(x)=0.
	\end{align*}
	It follows that $(-\omega_1(\delta),\omega_1(\delta))\cap \supp(\mu_1)=\emptyset$.
\end{proof}

Since the definitions of $\Theta_{\rm out}$ and $\Theta_{\rm in}$ are symmetric with respect to $\mu_1$ and $\mu_2,$ we have the following corollary due to a similar proof for Proposition \ref{prop:no-atom-omega_1}, for which we omit the details. 

\begin{corollary}\label{cor:no-atom-omega_1}
Let $\mu_1, \mu_2$ be symmetric probability measures on $\mathbb{R}$ with compact support given by \eqref{eq:mu-z}, and denote $\mu=\mu_1\boxplus \mu_2$. Suppose that $0\notin \supp(\mu)$. If $z \in \Theta_{\rm in}$, then there exists $\delta' >0$, such that for any $x \in (-\delta', \delta')$, we have $\omega_2(x)\in\mathbb{R}$ and $\omega_2'(x)>1$. Moreover, 
$(-\omega_2(\delta'),\omega_2(\delta'))\cap \supp(\mu_2)=\emptyset$.
\end{corollary}

We now consider two arbitrary symmetric compactly supported probability measures on $\mathbb{R}$. Without causing ambiguity, we still call them $\mu_1$ and $\mu_2,$ and denote $\mu= \mu_1 \boxplus \mu_2.$ Let $F_{\mu}(z) = 1/ G_{\mu}(z)$ and $ R_{\mu}(z)$ be the R-transform of $\mu,$ and set $R_i (z) = R_{\mu_i}(z),$ $F_i (z) = F_{\mu_i}(z)$, and $G_i(z)=G_{\mu_i}(z)$ for $i=1, 2.$ 

\begin{theorem}
 \cite[Theorem 17 in Chapter 3]{MingoSpeicherBook}
 \label{thm:domain.R.transform}
 Let $\nu$ be a probability measure on $\mathbb{R}$ with support contained in the interval $[-r,r]$. Then,
 the Cauchy transform $G_\nu$ is univalent on $\{z: \vert z\vert >4r \}$ and $\{z: 0<|z|<1/(6r)\}\subset \{G_\nu(z): |z|>4r\}$. The $R$-transform $R_\nu$ of $\nu$ is defined on $\{z: \vert z\vert <1/(6r) \}$ such that $G_\nu(R_\nu(z)+1/z)=z$ for $0<\vert z\vert <1/(6r)$. Moreover, if $\{\kappa_n\}_n$ are the free cumulants of $\nu$, then for $|z|<1/(6r)$, $\sum_{n\geq 1}\kappa_n z^{n-1}$ converges to $R_\nu(z)$.  
\end{theorem}

\begin{definition}
Let $\mu_1$ and $\mu_2$ be two probability measures on $\mathbb{R}$ with compact support. Suppose that ${\rm \supp}(\mu_1) \subseteq [-s_1, s_1]$ and ${\rm \supp}(\mu_2) \subseteq [-s_2, s_2].$ For $z \in  \mathbb{C}$ such that $|G_1(z)|< 1/(6s_2)$, we define 
\begin{equation}\label{eq:varphi}
h_1(z) = z + R_2 (G_1 (z)). 
\end{equation}
For $z \in  \mathbb{C}$ such that $|G_2(z)|< 1/(6s_1)$, we define 
\begin{equation}\label{eq:psi}
h_{2}(z) = z + R_1 (G_2 (z)). 
\end{equation}
\end{definition}

In the following proposition, we show that $h_1$ and $h_{2}$ are local inverse functions of subordination functions in some neighborhood of the origin. Denote $H_i (z)= F_i(z)-z$ for $i=1,2$. Recall that $F_{\mu_1\boxplus\mu_2}(z)=F_1(\omega_1(z))=F_2(\omega_2(z))$. Then, we have
\[
 \omega_2(z)=z+F_1(\omega_1(z))-\omega_1(z)=z+H_1(\omega_1(z)).
\] 
The subordination relation \eqref{eq:subordination} can be rewritten as
\begin{equation}\label{eq:fix-H}
F_1 (\omega_1 (z)) = F_2 (z+ H_1 (\omega_1 (z))).
\end{equation}

By Theorem \ref{thm:domain.R.transform}, For $|z|$ sufficiently small $F_2^{\langle-1 \rangle} (z^{-1})$ is well defined; then for such $z$ we have
\begin{equation*}
R_2 (z) = F_2^{\langle-1\rangle} (z^{-1}) - z^{-1}. 
\end{equation*}
Suppose that $|G_1(\omega_1(z))| < 1/(6s_2)$ and $|z + H_1(\omega_1(z))| > 4s_2$, by \eqref{eq:fix-H} then we have
\begin{align*}
   R_2 (G_1(\omega_1(z)))&=F_2^{\langle-1\rangle} (F_1 (\omega_1 (z)))-F_1 (\omega_1 (z))\\
   &=z+ H_1 (\omega_1 (z))-F_1 (\omega_1 (z))\\
   &=z-\omega_1(z).
\end{align*}
Hence the subordination function $\omega_1(z)$ satisfies the following relation:
\begin{equation}\label{eq:fix-sub}
\omega_1(z) = z - R_2 (G_1(\omega_1(z)))
\end{equation}
for $z \in \mathbb{C}$ such that $|G_1(\omega_1(z))| < 1/(6s_2)$ and $|z + H_1(\omega_1(z))| > 4s_2.$ Similarly, we have 
\begin{equation}
\omega_2 (z) = z - R_1 (G_2(\omega_2(z)))
\end{equation}
for $z \in \mathbb{C}$ such that $|G_2(\omega_2(z))| < 1/(6s_1)$ and $|z + H_2(\omega_2(z))| > 4s_1.$

\begin{lemma}\label{lem:subordination}
Suppose ${\supp}(\mu_1) \subseteq [-s_1, s_1]$ and ${\supp}(\mu_2) \subseteq [-s_2, s_2]$.
\begin{enumerate}[{\rm (1)}]
\item Assume that $[-\eta, \eta] \cap \supp(\mu_1) = \emptyset.$ If $\eta < \min \{2^{-1/4}, 1/(s_1+s_2)\},$
then $h_1$ is well-defined for any $|z|< \eta^3.$ 
\item  Assume that $[-\eta', \eta'] \cap \supp(\mu_2) = \emptyset.$ If $\eta' < \min \{2^{-1/4}, 1/(s_1+s_2)\},$
then $h_{2}$ is well-defined for any $|z|< \eta'^3.$ 
\end{enumerate}
\end{lemma}

\begin{proof}
It is sufficient to show ${\rm (1)},$ since ${\rm (2)}$ can be obtained by a symmetric argument. We note that ${\supp}(\mu) \subseteq [-(s_1+s_2), s_1+s_2]$.
Since $\mu_1$ is symmetric, for any $|z| < \eta,$ 
\begin{align}\label{eq:G_1(z)}
    G_1(z) &= \int_{\mathbb{R}\setminus[-\eta,\eta]}\frac{1}{z-t}\,{\rm d}\mu_1(t)\nonumber\\
    &=-\sum_{m=1}^\infty \int_{\mathbb{R}\setminus[-\eta,\eta]}\frac{z^{2m-1}}{t^{2m}}\, {\rm d}\mu_1(t) \nonumber\\
    &= -z\int_{\mathbb{R}\setminus[-\eta,\eta]}\frac{1}{t^2}\,{\rm d} \mu_1(t)-\sum_{m=2}^\infty \int_{\mathbb{R}\setminus[-\eta,\eta]}\frac{z^{2m-1}}{t^{2m}}\, {\rm d}\mu_1(t).
\end{align}
Then we have for any $|z|< \eta/\sqrt{2},$
\begin{align}
    \vert G_1(z)  \vert &\leq \vert z \vert m_{-2}(\mu_1) +\vert z\vert^3 \sum_{m=2}^\infty \int_{\mathbb{R}\setminus[-\eta,\eta]}\frac{\vert z\vert^{2m-4}}{t^{2m}}\,{\rm d}\mu_1(t)\nonumber\\
    &\leq \vert z \vert m_{-2}(\mu_1) +\vert z\vert^3 \sum_{m=2}^\infty \frac{\vert z\vert^{2m-4}}{\eta^{2m}}\nonumber\\
    &= \vert z \vert m_{-2}(\mu_1)  +\vert z\vert^3 \frac{1}{\eta^2(\eta^2-\vert z\vert^2)}\label{eq:estimate-G-1-V0}\\
    & \leq \vert z \vert \eta^{-2} + 2 \vert z\vert^3  \eta^{-4}, \label{eq:estimate-G-1}
\end{align}
where we have used the fact that $m_{-2} (\mu_1) \leq \eta^{-2}.$

If $0<\eta < \min \{2^{-1/4}, 1/(12 (s_1+s_2))\},$ then $\eta^3<\eta/\sqrt{2}$. 
By \eqref{eq:estimate-G-1}, for any $|z| < \eta^3$, we have 
\begin{equation}
 \label{eq:estimate-G-1-Vb}
\begin{split}
  \vert G_1(z)  \vert \leq \eta + 2 \eta^{5} < 2 \eta < 1/(6(s_1+s_2))<1/(6s_2). 
\end{split}
\end{equation}
Therefore, by Theorem \ref{thm:domain.R.transform}, $h_1(z)$ is well defined for any $|z| < \eta^3.$
\end{proof}

\begin{proposition}\label{prop:no-atom-mu}
Let $\mu_1$, $\mu_2$ be any symmetric probability measure on $\mathbb{R}$ with compact support such that 
 ${\supp}(\mu_1) \subseteq [-s_1, s_1]$ and ${\supp}(\mu_2) \subseteq [-s_2, s_2]$. Denote $\mu = \mu_1\boxplus \mu_2$. Suppose that $[-\eta,\eta]\cap \mathrm{supp}(\mu_1) = \emptyset.$ Denote $\eta_0: = \min\{\eta, 2^{-1/4}, 1/(12(s_1+s_2))\}$. If $1- m_2(\mu_2)m_{-2}(\mu_1)>0$, then there exists $\delta>0$ such that  $h_1$ is strictly increasing on $(-\gamma/2, \gamma/2)$ and $[-h_1(\gamma/4), h_1(\gamma/4)] \cap \supp(\mu)=\emptyset$,
where 
\[\gamma = \min\left\{\eta_0^3, \frac{\delta \eta^2}{2}, \frac{1-m_2(\mu_2)m_{-2}(\mu_1)}{4K}\right\},\]
and 
\[
  K=2 m_2(\mu_2)\eta^{-4}+  8 (\vert \kappa_4(\mu_2)\vert+1) \eta^{-6}.
\]
\end{proposition}

\begin{proof}
By Theorem \ref{thm:domain.R.transform}, we have $R_2(z) = \sum_{n\geq 1}\kappa_n(\mu_2)z^{n-1}$ for all $\vert z\vert<1/(6s_2)$, where $(\kappa_n(\mu_2))_{n\geq 1}$ are free cumulants. Since $\mu_2$ is symmetric, we have
\[R_2(z) = \kappa_2(\mu_2)z+\sum_{n=2}^\infty \kappa_{2n}(\mu_2)z^{2n-1}.\]
Therefore, there exists $\delta>0$ depending only on $\mu_2$ such that 
\begin{align}
    \vert R_2(z) - \kappa_2(\mu_2)z\vert &\leq \vert z\vert^3\left(\vert\kappa_4(\mu_2)\vert+\vert z\vert^2\sum_{n=3}^\infty \vert \kappa_{2n}(\mu_2)\vert \vert z\vert^{2n-6}\right)\nonumber\\
    &\leq (\vert\kappa_4(\mu_2)\vert+1)\vert z\vert^3\label{eq.R2.estimate}
\end{align}
for all $\vert z\vert\leq \delta$. By (\ref{eq:estimate-G-1-V0}), for any $\vert z\vert<\eta/\sqrt{2}$, we have 
\begin{equation} \label{eqn:G1.estimate}
\begin{split}
   |G_1(z)| &\leq \vert z \vert m_{-2}(\mu_1)  +\vert z\vert^3 \frac{1}{\eta^2(\eta^2-\vert z\vert^2)}\\
 &  \leq \vert z\vert (m_{-2}(\mu_1)+1/\eta^2)
     \leq 2\vert z\vert/\eta^2,
     \end{split}
\end{equation}
where we used the inequality $m_{-2}(\mu_1)\leq 1/\eta^2$. 
By \eqref{eq:estimate-G-1-Vb}, for any $\vert z\vert<\eta_0^3$, we have $\vert G_1(z) \vert \leq  1/(6s_2)$. We note that $\eta_0^3 \leq\eta_0/\sqrt{2}\leq \eta/\sqrt{2}$.
Therefore, 
\begin{equation}\label{eq:absolute-G1}
\vert G_1(z) \vert \leq \min\{ \delta, 1/(6s_2) \}  
\end{equation}
for any $|z|< \min\left\{\eta_0^3, \delta\eta^2/2\right\}$.
We then estimate
\begin{align}
    \vert G_1(z) + zm_{-2}(\mu_1) \vert \leq   \vert z\vert^3 \frac{1}{\eta^2(\eta^2-\vert z\vert^2)}\leq 2\eta^{-4}\vert z\vert^3\label{eq.G1.estimate}
\end{align}  
for $|z|< \min\left\{\eta_0^3, \delta\eta^2/2\right\}$.

Since $\mu_2$ is symmetric, $\kappa_2(\mu_2) = m_2(\mu_2)$. We then write
\begin{align}
    &h_1(z) -(z-m_2(\mu_2)m_{-2}(\mu_1)z)\nonumber\\
     &=[z+R_2(G_1(z))] - [z-m_2(\mu_2)m_{-2}(\mu_1)z]\nonumber\\
    &= [R_2(G_1(z)) -m_2(\mu_2)G_1(z)]+m_2(\mu_2)[G_1(z)+m_{-2}(\mu_1)z].\label{eq.varphi.estimate}
\end{align}
For $|z|< \min\left\{\eta_0^3, \delta\eta^2/2\right\}$, we then have 
\begin{align*}
    \vert R_2(G_1(z))-m_2(\mu_2)G_1(z)\vert &\leq (\vert\kappa_4(\mu_2)\vert+1)\vert G_1(z)\vert^3\\
    &\leq 8 (\vert \kappa_4(\mu_2)\vert+1) \eta^{-6} \vert z\vert^3
\end{align*}
by \eqref{eq.R2.estimate} and \eqref{eqn:G1.estimate}. Thus, \eqref{eq.varphi.estimate} gives
\begin{align}
\label{eq.varphi.estimate1}
    &\vert h_1(z) -(z-m_2(\mu_2)m_{-2}(\mu_1)z)\vert \nonumber \\
    &\qquad\leq \Big[ 8 (\vert \kappa_4(\mu_2)\vert+1) \eta^{-6} + 2 m_2(\mu_2)\eta^{-4} \Big] \vert {z}\vert^3 = K \vert z \vert^3
\end{align}
where $K=2 m_2(\mu_2)\eta^{-4}+  8 (\vert \kappa_4(\mu_2)\vert+1) \eta^{-6}$
for any $|z|< \min\left\{\eta_0^3, \delta\eta^2/2\right\}$. 

Let $f(z) =  h_1(z) -\Big(1-m_2(\mu_2)m_{-2}(\mu_1)\Big)z$. Then $f$ is analytic on the disk with radius equal to 
$r= \min\left\{\eta_0^3, \delta\eta^2/2\right\}$.
By the Cauchy integral formula,
\[f'(z) = \frac{1}{2\pi i}\int_{\mathcal{C}} \frac{f(\zeta)}{(\zeta-z)^2}\, {\rm d} \zeta\]
where the contour $\mathcal{C}$ is the circle centered at the origin with radius $\gamma$, where 
\[\gamma =\min\left\{\eta_0^3,\frac{\delta \eta^2}{2}, \frac{1-m_2(\mu_2)m_{-2}(\mu_1)}{4K}\right\}.\]
Therefore, for any $|z| < \gamma/2,$ \eqref{eq.varphi.estimate1} gives us an estimate of $h_1'$ by
\begin{align}
   & \vert h_1'(z) - (1-m_2(\mu_2)m_{-2}(\mu_1))\vert \nonumber \\ 
   & \quad \quad \quad = \vert f'(z)\vert \leq \frac{1}{2\pi}\left\vert\int_{\mathcal{C}} \frac{ f(\zeta)}{ (\zeta-z)^2}\, {\rm d} \zeta\right\vert \nonumber\\
   & \quad \quad \quad <1-m_2(\mu_2)m_{-2}(\mu_1), \label{eqn:estimate.derivative}
\end{align}
where we have used the fact that $\vert \zeta \vert = \gamma$ and 
\begin{equation*}
\frac{\vert f(\zeta)\vert}{\vert \zeta-z\vert^2} \leq \frac{K |\zeta|^3}{\vert \zeta-z\vert^2} <  4K \gamma \leq 1-m_2(\mu_2)m_{-2}(\mu_1).
\end{equation*}
provided that $\vert z \vert < \gamma/2.$ 
We note that $h_1(x)$ is well-defined and $h_1(x)\in\mathbb{R}$ for any $x\in (-\eta_0^3,\eta_0^3)$ by Lemma \ref{lem:subordination}.
Then inequality \eqref{eqn:estimate.derivative} guarantees $h_1'(x)>0$ for $x\in (-\gamma/2, \gamma/2)$; in particular, $h_1$ is strictly increasing on $(-\gamma/2, \gamma/2)$. 

Now we will prove the following subordination relation
\begin{equation}\label{eq:subordination-phi}
G_\mu (h_1(z)) = G_1(z), \;\; |z| < \gamma/2.
\end{equation}
The idea more or less comes from \cite[Proposition 3.5]{Serban2021}. For any $|z|< \gamma/2,$ it follows from \eqref{eqn:estimate.derivative} that 
\begin{equation*}
\begin{split}
\vert R_2'(G_1(z)) \vert & =\vert h_1'(z) - 1\vert \\
& \leq \vert h_1'(z) -(1-m_2(\mu_2)m_{-2}(\mu_1)) \vert + m_2(\mu_2)m_{-2}(\mu_1) \\
& <1
\end{split}
\end{equation*}
We claim that $h_1$ is injective for $\vert z\vert<\gamma/2$. Fix a $z$ such that $\vert z\vert<\gamma/2$. Consider the following map 
\begin{equation*}
  \rho(v) = h_1(z) - R_2(G_1(v)), \;\; |v| < \gamma/2.
\end{equation*}
Note that $z$ is a fixed point of $\rho$. Since $\vert \rho'\vert<1,$ $z$ is an attracting fixed point for this map. Therefore, $z$ is the unique fixed point of $\rho$. In particular, if $h_1(v) =v+R_2(G_1(v))= h_1(z)$ for some $v$ such that $\vert v\vert<\gamma/2$, then $v$ is a fixed point of $\rho$ and $v=z$. This shows that $h_1$ is injective for $\vert z\vert<\gamma/2$. Hence, $h_1^{-1}$ exists on $\{h_1 (z): |z| < \gamma/2 \}$.

Now we prove $\omega_1 = h_1^{-1}$ on $\{h_1 (z): |z| < \gamma/2 \}$. First, by Lemma \ref{lem:descr-subordination}, we note that $\omega_1(0) = 0$, $h_1(0) = 0.$
Thus, there is a small enough neighborhood $V$ of $0$ such that $V\subset \{h_1 (z): |z| < \gamma/2 \}$ and $\vert\omega_1(z)\vert < \gamma/2$ for all $z\in V$. By \eqref{eq:absolute-G1}, we have
\begin{equation*}
\vert G_1(\omega_1(z)) \vert < 1/(6s_2) 
\end{equation*}
for all $z \in V$. Moreover, denote $b =\omega_1(z),$ $z \in V.$ We have 
\begin{equation*}
\begin{split}
\vert h_1(b) + H_1 (b) \vert & = \frac{\vert 1+ R_2 (G_1(b)) \cdot G_1(b) \vert}{\vert G_1(b) \vert} \\
& \geq \frac{\left\vert 1- \vert R_2(G_1(b)) \vert \cdot \vert G_1(b) \vert \right\vert}{\vert G_1(b) \vert}.
\end{split}
\end{equation*}
Note that $|b|< \gamma/2,$ by \eqref{eq.R2.estimate} and \eqref{eqn:G1.estimate}, we have 
\begin{equation*}
\begin{split}
\vert R_2(G_1(b)) \vert \cdot \vert G_1(b) \vert & \leq (1+ \vert \kappa_4(\mu_2)\vert) \vert G_1(b) \vert^4 + \kappa_2 (\mu_2) \vert G_1(b) \vert^2\\
& = \frac{8(1+ \vert \kappa_4(\mu_2)\vert)} {\eta^6} \cdot \frac{\vert G_1(b) \vert^4 \eta^6}{8} + \frac{2 \kappa_2(\mu_2)}{\eta^4}\cdot   \frac{\vert G_1(b) \vert^2 \eta^4}{2} \\
& \leq \frac{8(1+ \vert \kappa_4(\mu_2)\vert)} {\eta^6} \vert b \vert^2 + \frac{2 \kappa_2(\mu_2)}{\eta^4} 2 \vert b \vert^2 \\
& < \frac{\gamma^2}{4} \cdot K <\frac{1}{16},
\end{split}
\end{equation*}
where we have used fact that $\gamma<1$ and $4 \gamma K < 1-m_2(\mu_2)m_{-2}(\mu_1)<1.$ Thus, $\vert h_1(b) + H_1 (b) \vert> \frac{15}{16\vert G_1(b) \vert} > 4s_2.$ So \eqref{eq:fix-sub} holds for all $z \in V.$ Therefore, we have 
\begin{equation*}
h_1(\omega_1(z))= z, \;\; z\in V.
\end{equation*}
This shows that $h_1^{-1} = \omega_1$ on $V$. By an analytic continuation, $h_1^{-1} = \omega_1$ on $\{h_1 (z): |z| < \gamma/2 \}$. Finally, since $\omega_1$ is the subordination function, we have
\begin{equation*}
G_\mu (h_1(z)) = G_1(\omega_1 (h_1(z))) = G_1(z)
\end{equation*}
for all $z$ such that $\vert z\vert<\gamma/2$.

So we have $\omega_1(b) = h_1^{-1} (b),$ and in particular, $\omega_1(h_1(z)) =z.$ Therefore, we have
\begin{equation*}
G_\mu (h_1(z)) = G_1(\omega_1 (h_1(z))) = G_1(z).
\end{equation*}

Note that $h_1 (0)=0$ and $h_1$ is an odd function on $(-\eta_0^3, \eta_0^3).$ Since $h_1$ is strictly increasing on $(-\gamma/2, \gamma/2)$, we have 
$h_1 ((-\gamma/2, \gamma/2))= (-h_1 (\gamma/2), h_1(\gamma/2) ).$
Then by the subordination relation \eqref{eq:subordination-phi} we have  
\begin{equation}
\lim_{y \downarrow 0} {\rm Im} G_\mu(h_1(x) + iy) =  {\rm Im} G_1(x) = 0
\end{equation}
for any $x \in (-\gamma/2, \gamma/2).$ Then our result follows by the Stieltjes inversion formula. 
\end{proof}

Similar to Corollary \ref{cor:no-atom-omega_1}, Proposition \ref{prop:no-atom-mu} has the following corollary. 

\begin{corollary}\label{cor:no-atom-mu}
Let $\mu_1$, $\mu_2$ be any symmetric probability measure on $\mathbb{R}$ with compact support such that 
 ${\supp}(\mu_1) \subseteq [-s_1, s_1]$ and ${\supp}(\mu_2) \subseteq [-s_2, s_2]$. Denote $\mu = \mu_1\boxplus \mu_2$. Suppose that $[-\eta',\eta']\cap \mathrm{supp}(\mu_2) = \emptyset.$ Denote $\eta_0': = \min\{\eta', 2^{-1/4}, 1/(12(s_1+s_2))\}$. If $1- m_2(\mu_1)m_{-2}(\mu_2)>0$, then there exists $\delta'>0$ such that  $h_{2}$ is strictly increasing on $(-\gamma'/2, \gamma'/2)$ and $[-h_{2}(\gamma'/4), h_{2}(\gamma'/4)] \cap \supp(\mu)=\emptyset$,
where 
\[\gamma' = \min\left\{\eta_0'^3, \frac{\delta' \eta'^2}{2}, \frac{1-m_2(\mu_1)m_{-2}(\mu_2)}{4K'}\right\},\]
and 
\[
  K'=2 m_2(\mu_1)\eta'^{-4}+  8 (\vert \kappa_4(\mu_1)\vert+1) \eta'^{-6}.
\]
\end{corollary}

\begin{remark}
Similar to \eqref{eq:subordination-phi}, we can also prove the following subordination relation
\begin{equation}
G_\mu (h_2(z)) = G_2(z), \;\; |z| < \gamma'/2.
\end{equation}
\end{remark}

Let $\mu_1, \mu_2$ be symmetric probability measures on $\mathbb{R}$ with compact support given by \eqref{eq:mu-z}, and denote $\mu=\mu_1\boxplus \mu_2$. Let 
\begin{equation}\label{eq:a_n}
\mu_1^{(n)} = \tilde{\mu}_{|a_n -z|},
\end{equation}
where $\sup_n\|a_n\| \leq C$ and $a_n$ converges to $a$ in $*$-moments as $n\rightarrow\infty$.

Suppose that there exists $\eta>0$ such that 
\begin{equation*}
[-\eta, \eta] \cap \supp (\mu_1) = [-\eta, \eta] \cap \supp (\mu_1^{(n)})= \emptyset
\end{equation*}
for all large $n.$ Denote $G_{1, n} (z)= G_{\mu_1^{(n)}} (z),$ and 
\begin{equation}
h_{1,n} (z) := z + R_2 (G_{1, n}(z))
\end{equation}
for $z$ in a neighborhood of the origin.

\begin{theorem}\label{thm:support}
Let $\mu_{1} = \tilde{\mu}_{|a -z|}$, $\mu_2 = \tilde{\mu}_{|T|}$, and $\mu=\mu_1\boxplus \mu_2$. Let $(\mu_1^{(n)})_{n \geq 1}$ be a sequence of symmetric compactly supported probability measures given by \eqref{eq:a_n}.
We assume that
\begin{enumerate}[{\rm (a)}]
\item Assumption \ref{assump:Xa} holds: $\supp (\mu_{a +T}) = {\rm spec} (a+T)$;
\item $z \in \Theta_{\rm out}=\left\{ z \in \mathbb{C}: m_2 (\mu_2) < \frac{1}{m_{-2} (\mu_1)} \right\}$;
\item There exists $\eta>0$ such that $[-\eta, \eta] \cap \supp (\mu_1^{(n)})= \emptyset$ for all large $n.$ 
\end{enumerate}
Then there exists $\varepsilon>0$ such that $[-\varepsilon,\varepsilon]\cap \supp (\mu_1^{(n)} \boxplus \mu_2) = \emptyset$ for all large $n$.
\end{theorem}

\begin{proof}
Since $z \in \Theta_{\rm out}$, then Theorem \ref{thm:supportOFaPlusT} implies that $0 \notin \supp (\mu_1 \boxplus \mu_2)$ by assumption (a). By Proposition \ref{prop:no-atom-omega_1}, $0 \notin \supp (\mu_1)$, so we may assume that $[-\eta, \eta] \cap \supp(\mu_1) = \emptyset.$ Note that $\supp (\mu_1^{(n)})$ are uniformly bounded. Then by Lemma \ref{lem:subordination}, we may choose $\eta_0$ small enough such that $h_{1,n}$ and $h_1$ are well-defined on $\{z \in \mathbb{C}: |z| < \eta_0^3\}.$ 
Note that $\mu_1^{(n)} \rightarrow \mu_1$ weakly as $n \rightarrow \infty.$ 
Hence, 
\[
  \lim_{n\rightarrow\infty}h_{1,n}(z)=h_1(z)
\]
uniformly on compact subsets of $\{z \in \mathbb{C}: |z| < \eta_0^3\}.$ 
In addition, we have 
\begin{equation*}
m_2(\mu_2) m_{-2}(\mu_1^{(n)}) \to m_2(\mu_2) m_{-2}(\mu_1)
\end{equation*} 
as $n \to \infty.$ Therefore, $m_2(\mu_2) m_{-2}(\mu_1^{(n)}) <1$ for $n$ sufficiently large. Hence, by Assumption $(b)$ and Proposition \ref{prop:no-atom-mu}, there exists 
$\gamma_n >0,$ such that $h_{1,n}$ is strictly increasing on $(-\gamma_n/2, \gamma_n/2),$ where
\begin{equation*}
\gamma_n = \min\left\{\eta_0^3,\frac{\delta \eta^2}{2}, \frac{1-m_2(\mu_2)m_{-2}(\mu_1^{(n)})}{4K}\right\}.
\end{equation*}
So we can conclude that $h_{1,n}$ is strictly increasing on $(-\gamma/2, \gamma/2)$ for all large $n$, where 
\begin{equation}
\gamma =  \min\left\{\eta_0^3,\frac{\delta \eta^2}{2}, \frac{1-m_2(\mu_2)m_{-2}(\mu_1)}{4K}\right\}.
\end{equation}
Moreover, we have the following subordination relations: 
\begin{equation}\label{eq:subordination-n}
\begin{split}
G_{\mu_1^{(n)} \boxplus \mu_2} (h_{1,n} (z)) = G_{1, n} (z), \; |z|< \gamma/2
\end{split}
\end{equation}
for all large $n.$ 
It implies that 
\begin{equation}\label{eq:no-support-varphi_n}
[-h_{1,n} (\gamma/4), h_{1,n}( \gamma/4)] \cap \supp (\mu_1^{(n)} \boxplus \mu_2) = \emptyset 
\end{equation}
for $n$ sufficiently large.

By Proposition \ref{prop:no-atom-mu}, then $h_1$ is also strictly increasing on $(-\gamma/2, \gamma/2).$ Now, since $h_{1,n} (\gamma/4)\to h_1(\gamma/4)$ as $n\to\infty$, the fact that $h_1(\gamma/4)>h_1(0) =0$ shows that there exists $ \varepsilon >0$ such that $h_{1,n}(\gamma/4)>\varepsilon$ for all large $n$. Then we have
    \[[-\varepsilon,\varepsilon]\subset h_{1,n} ([-\gamma/4,\gamma/4])\]
    for all large $n$. Hence, our result follows by \eqref{eq:no-support-varphi_n}.
\end{proof}

Suppose that there exists $\eta'>0$ such that 
\begin{equation*}
[-\eta', \eta'] \cap \supp (\mu_2) = \emptyset.
\end{equation*}
Denote $R_{1, n} (z)= R_{\mu_1^{(n)}} (z),$ and 
\begin{equation}
h_{2,n} (z) := z + R_{1, n} (G_2(z))
\end{equation}
for $z$ in a neighborhood of the origin.

\begin{theorem}\label{thm:support-in}
Let $\mu_{1} = \tilde{\mu}_{|a -z|}$, $\mu_2 = \tilde{\mu}_{|T|}$, and $\mu=\mu_1\boxplus \mu_2$. Let $(\mu_1^{(n)})_{n \geq 1}$ be a sequence of symmetric compactly supported probability measures given by \eqref{eq:a_n}.
We assume that
\begin{enumerate}[{\rm (a)}]
\item Assumption \ref{assump:Xa} holds: $\supp (\mu_{a +T}) = {\rm spec} (a+T)$;
\item $z \in \Theta_{\rm in}=\left\{ z \in \mathbb{C}: m_2 (\mu_1) < \frac{1}{m_{-2} (\mu_2)} \right\}.$
\end{enumerate}
Then there exists $\varepsilon'>0$ such that $[-\varepsilon',\varepsilon']\cap \supp (\mu_1^{(n)} \boxplus \mu_2) = \emptyset$ for all large $n$.
\end{theorem}

\begin{proof}
Since $z \in \Theta_{\rm in}$, then Theorem \ref{thm:supportOFaPlusT} implies that $0 \notin \supp (\mu_1 \boxplus \mu_2)$ by assumption (a). By Corollary \ref{cor:no-atom-omega_1}, $0 \notin \supp (\mu_2)$, so we may assume that $[-\eta', \eta'] \cap \supp(\mu_2) = \emptyset.$ Note that $\supp (\mu_1^{(n)})$ are uniformly bounded. Then by Lemma \ref{lem:subordination}, we may choose $\eta_0'$ small enough such that $h_{2,n}$ and $h_{2}$ are well-defined on $\{z \in \mathbb{C}: |z| < \eta_0'^3\}.$ Note that $\mu_1^{(n)} \rightarrow \mu_1$ weakly as $n \rightarrow \infty.$ 
Hence, 
\[
  \lim_{n\rightarrow\infty}h_{2,n}(z)=h_{2}(z)
\]
uniformly on compact subsets of $\{z \in \mathbb{C}: |z| < \eta_0'^3\}.$ 
In addition, we have 
\begin{equation*}
m_{2}(\mu_1^{(n)}) m_{-2}(\mu_2) \to  m_{2}(\mu_1)m_{-2}(\mu_2)
\end{equation*} 
as $n \to \infty.$ Therefore, $m_{2}(\mu_1^{(n)}) m_{-2}(\mu_2)  <1$ for $n$ sufficiently large. Hence, by Assumption $(b)$ and Corollary \ref{cor:no-atom-mu}, there exists 
$\gamma'_n>0,$ such that $h_{2,n}$ is strictly increasing on $(-\gamma'_n/2, \gamma'_n/2),$ where
\begin{equation*}
\gamma_n' = \min\left\{\eta_0'^3,\frac{\delta' \eta'^2}{2}, \frac{1-m_{2}(\mu_1^{(n)}) m_{-2}(\mu_2) }{4K_n'}\right\}
\end{equation*}
with $K_n'= 2 m_2(\mu_1^{(n)})\eta'^{-4}+  8 (\vert \kappa_4(\mu_1^{(n)})\vert+1) \eta'^{-6}.$ Note that $m_2(\mu_1^{(n)}) \to m_2(\mu_1)$ and $\kappa_4(\mu_1^{(n)}) \to \kappa_4(\mu_1)$ as 
$n \to \infty.$ So we can conclude that $h_{2,n}$ is strictly increasing on $(-\gamma'/2, \gamma'/2)$ for all large $n$, where 
\begin{equation}
\gamma' =  \min\left\{\eta_0'^3,\frac{\delta' \eta'^2}{2}, \frac{1-m_2(\mu_1)m_{-2}(\mu_2)}{4K'}\right\}.
\end{equation}
Moreover, we have the following subordination relations: 
\begin{equation}\label{eq:subordination-n-in}
\begin{split}
G_{\mu_1^{(n)} \boxplus \mu_2} (h_{2,n} (z)) = G_{2} (z), \; |z|< \gamma'/2
\end{split}
\end{equation}
for all large $n.$ it implies that 
\begin{equation}\label{eq:no-support-psi_n}
[-h_{2,n} (\gamma'/4), h_{2,n}( \gamma'/4)] \cap \supp (\mu_1^{(n)} \boxplus \mu_2) = \emptyset 
\end{equation}
for $n$ sufficiently large.

By Corollary \ref{cor:no-atom-mu}, then $h_{2}$ is also strictly increasing on $(-\gamma'/2, \gamma'/2).$ Now, since $h_{2,n} (\gamma'/4)\to h_{2} (\gamma'/4)$ as $n\to\infty$, the fact that $h_{2}(\gamma'/4)>h_{2}(0) =0$ shows that there exists $ \varepsilon'>0$ such that $h_{2,n}(\gamma'/4)>\varepsilon'$ for all large $n$. Then we have
    \[[-\varepsilon',\varepsilon']\subset h_{2,n} ([-\gamma'/4,\gamma'/4])\]
    for all large $n$. Hence, our result follows by \eqref{eq:no-support-psi_n}.
\end{proof}

\begin{lemma}\label{lem:invertible-mu_{omega, z}-0}
The function 
\begin{equation*}
 p_1: z \mapsto m_2(|T|) \cdot m_{-2}(|a-z|)
\end{equation*} 
is continuous in $\Theta_{\rm out}.$

Moreover, the function 
\begin{equation*}
p_2: z \mapsto m_{2}(|a-z|) \cdot m_{-2}(|T|) 
\end{equation*} 
is continuous in $\Theta_{\rm in}.$
\end{lemma}

\begin{proof}
Recall for any $z \in \Theta_{\rm out},$
\begin{equation*}
m_{-2}(|a-z|) = \tau (|a-z|^{-2}) = \int_{\mathbb{R}} \lambda^{-2} {\rm d} \mu_{1}(\lambda). 
\end{equation*}
By Proposition \ref{prop:no-atom-omega_1}, $0\notin \supp(\mu_{1}).$ So for any $z_1, z_2 \in \Theta_{\rm out},$ 
\begin{equation*}
\begin{split}
\left\vert m_{-2}(|a-z_1|)- m_{-2}(|a-z_2|) \right\vert & = \int_{\mathbb{R}} \lambda^{-2} {\rm d} (\mu_{1, z_1}- \mu_{1, z_2})(\lambda)\\
& \leq C {\rm d}_{L} (\mu_{1, z_1}, \mu_{1, z_2}), 
\end{split}
\end{equation*}
where $\mu_{i, z_i}= \tilde{\mu}_{|a-z_i|}, i=1,2$ and $C$ is a finite constant depending on $z_1$ and $z_2.$ Then by the standard inequality for the spectral measure, we have 
\begin{equation*}
\begin{split}
{\rm d}_{L} (\mu_{1, z_1}, \mu_{1, z_2}) &\leq \| |a-z_1|^2- |a-z_2|^2\|\\
& \leq \left\vert |z_1|^2 - |z_2|^2 \right\vert + |z_1- z_2| (\|a^*\| + \|a\|). 
\end{split}
\end{equation*}
This proves the continuity of $p_1$. The proof of the continuity of $p_2$ is similar.
\end{proof}

For any $\omega \neq 0,$ we denote $\mu_{2, \omega}: = \tilde{\mu}_{|\omega T|}.$ We have the following lemma.

\begin{lemma}\label{lem:mu_{omega, z}-0}
Let $\Gamma$ be a compact subset in $\Theta_{\rm out}.$  Then there exists $\rho >1$ such that for any $\omega \in \mathbb{C}$ such that 
$|\omega| \leq \rho$ and any $z \in \Gamma,$ we have $0 \notin {\rm supp} (\mu_{1} \boxplus \mu_{2, \omega}).$

Moreover, let $\Gamma$ be a compact subset in $\Theta_{\rm in}.$  Then there exists $\rho' >1$ such that for any $\omega \in \mathbb{C}$ such that 
$|\omega|^{-1} \leq \rho'$ and any $z \in \Gamma,$ we have $0 \notin {\rm supp} (\mu_{1} \boxplus \mu_{2, \omega}).$
\end{lemma}

\begin{proof}
By Lemma \ref{lem:invertible-mu_{omega, z}-0}, the function $p_1(z) = m_2(|T|) \cdot m_{-2}(|a-z|)$ is continuous on $\Theta_{\rm out},$ it attains its lowest upper bound on the compact set $\Gamma$ (recall $p_1(z) < 1$ for any $z \in \Theta_{\rm out}$). Thus, there exists $0 < \gamma <1$ such that for any $z \in \Gamma,$ we have $0 \leq p_1(z) < 1-\gamma.$ Hence, if $|\omega| \leq \frac{1}{\sqrt{1-\gamma}}$ then $m_2(|\omega T|) \cdot m_{-2}(|a-z|)< 1$. Then the first part of our result follows by Proposition \ref{prop:no-atom-mu}. 

Combining Corollary \ref{cor:no-atom-mu}, the second part of our result follows by an argument similar to the above proof.
\end{proof}


\section{No outliers; proof of Theorem \ref{thm:no-outlier}}\label{sec:no-outlier}

Recall that ${\rm spec} (a+ T) = \{z \in \mathbb{C}: 0 \in \supp (\mu_z)\},$ where $\mu_z$ is the distribution of $(a+T-z)(a+T-z)^*.$ Theorem \ref{thm:no-outlier} is equivalent to the following proposition.

\begin{proposition}\label{prop:no-outilier}
Let $\Gamma$ be a compact set with a continuous boundary that satisfies the assumptions of cases $(a)$ or $(b)$ of Theorem \ref{thm:no-outlier}. 
Then, for any $z \in \Gamma$, there exists $\gamma_z>0,$ such that almost surely, for all large $n$, $s_n(M_n - z)\geq \gamma_z$. Consequently, there exists 
$\gamma_\Gamma$ such that almost surely, for all large $n$, $\inf_{z \in \Gamma} s_n (M_n - z) \geq \gamma_\Gamma.$
\end{proposition}

For case $(a),$ we adapt the approach for the proof of \cite[Proposition 5.1]{Serban2021}. The proof  is based on the two following key results, whose proofs will be given in the next subsections. For $n \geq 1,$ let $\{ a_n, a\}$ be noncommutative random variables in $(\mathcal{A}, \tau)$ such that $\{a_n, a\},$ $T$ are freely independent. 

\begin{proposition}\label{prop:no-outlier-1}
We assume that $a_n \rightarrow a$ in $*$-moments and that there exists a constant $C>0$ such that $\sup_n \Vert a_n \Vert \leq C.$ Moreover, 
\begin{enumerate}[{\rm (1)}]
\item We fix $z \in \Theta_{\rm out}$ such that $\vert a + T-z\vert \geq \delta_z$ for some $\delta_z >0.$
\item We assume that there exists $n_{\delta_{z}}$ such that if $n \geq n_{\delta_{z}},$ then $\vert a_n-z \vert \geq \delta_{z}.$ 
\end{enumerate}
Then, there exists $\varepsilon_{z} >0$ and $n_{\varepsilon_{z}}>0$ such that if $n \geq n_{\varepsilon_{z}},$ then $\vert a_n + T-z \vert \geq \varepsilon_{z}$.
\end{proposition}

\begin{proposition}\label{prop:no-outlier-2}
Assume that the distribution of $a_n$ in $(\mathcal{A}, \tau)$ coincides with the distribution of $A_n$ in $(M_n(\mathbb{C}), {\rm tr}_n).$ For any $z \in \mathbb{C},$ let $\mu_{z}^{(n)}$ be the distribution of $( a_n + T -z)( a_n + T-z)^*$ with respect to $\tau.$ If there exist $0<\lambda_1 <\lambda_2$ and $0<\delta<(\lambda_2-\lambda_1)/2$ such that $(\lambda_1- \delta, \lambda_2+ \delta) \cap \supp (\mu_{z}^{(n)}) = \emptyset$ for $n$ sufficiently large, then for all large $n$, almost surely we have
\begin{equation*}
{\rm spec} ((M_n-z)(M_n-z)^* ) \cap [\lambda_1, \lambda_2] = \emptyset.
\end{equation*}
\end{proposition}

Similarly, the proof of case $(b)$ of Proposition \ref{prop:no-outilier} is based on the following Corollary. 

\begin{corollary}\label{cor:no-outlier-1}
We assume that $\{a_n\}$ is a sequence of operators in $\mathcal{A}$ such that $a_n \rightarrow a$ in $*$-moments. Suppose that $\sup_n \Vert a_n \Vert \leq C$ for some constant $C>0$. Moreover, 
we fix $z \in \Theta_{\rm in}$ such that $\vert a + T-z\vert \geq \delta'_z$ for some $\delta'_z >0.$
Then, there exists $\varepsilon_{z}' >0$ and $n_{\varepsilon_{z}'}>0$ such that if $n \geq n_{\varepsilon_{z}'},$ then $\vert a_n +T-z \vert \geq \varepsilon_{z}'.$
\end{corollary}

\begin{proof}[Proof of Proposition \ref{prop:no-outilier}]
Let $\mu_z^{(n)}$ be the distribution of $(a_n + T -z)(a_n + T-z)^*.$ For $z \in \Theta_{\rm out},$ 
by the assumption and Proposition \ref{prop:no-outlier-1}, there exists $\varepsilon_z >0,$ such that for 
sufficiently large $n,$
\begin{equation*}
[0, \varepsilon_z] \cap  \supp(\mu_{z}^{(n)}) = \emptyset.
\end{equation*}
Then, by Proposition \ref{prop:no-outlier-2}, there exists $\gamma_z >0$ such that almost surely for all large $n,$ there is no singular value of $M_n-z$ in $[0, \gamma_z], i.e.,$
$s_n (M_n -z) > \gamma_z.$ 

Similarly, for any $z \in \Theta_{\rm in},$ Corollary \ref{cor:no-outlier-1} and Proposition \ref{prop:no-outlier-1} imply that 
there exists $\gamma_z'>0,$ such that almost surely $s_n (M_n -z) > \gamma_z'$ for all large $n.$

The rest of the assertion of the proposition is due to a compactness argument and the fact that $z \mapsto s_n(M_n-z)$ is Lipschitz continuous. 
\end{proof}


\subsection{Proof of Proposition \ref{prop:no-outlier-1} and Corollary \ref{cor:no-outlier-1}}

For any element $x \in (\mathcal{A}, \tau),$ $x$ is invertible if and only if $x^* x$ and $x x^*$ are invertible. Since $\tau$ is faithful and tracial, we have that
$0 \notin \supp(\mu_{x^*x})$ if and only if $x$ is invertible. Moreover, $|x|^2\geq \delta >0$ if and only if $[0, \delta] \cap \supp(\mu_{x^*x}) = \emptyset.$

\begin{lemma}\label{lem:lower-bound}
Let $x$ be a self-adjoint element in $(\mathcal{A}, \tau),$ denote $\tilde{\mu}_x$ by the symmetrization of $\mu_x.$ Then $x$ is invertible if and only if $0 \notin \supp (\tilde{\mu}_x),$ and moreover, 
$|x|^2\geq \delta >0$ if and only if $[-\sqrt{\delta}, \sqrt{\delta}] \cap \supp (\tilde{\mu}_{x}) = \emptyset.$
\end{lemma}

\begin{proof}
It is a direct consequence of the fact that 
\begin{equation*}
{\rm spec} \left( \begin{bmatrix} 0& x\\ x & 0\end{bmatrix}\right) = \supp (\tilde{\mu}_x).
\end{equation*}
\end{proof}

\begin{proof}[Proof of Proposition \ref{prop:no-outlier-1} and Corollary \ref{cor:no-outlier-1}]

Following the notations in Section \ref{sec:free-prob}, we denote
\begin{equation}
\mu_1 =\tilde{\mu}_{|a-z|}  \; \text{and} \; \mu_2 = \tilde{\mu}_{|T|}.
\end{equation}
For any $z \in \Theta_{\rm out},$ since $\vert a + T  -z\vert$ is invertible, there exists $\delta_{z} >0$ such that $\vert a + T  -z\vert \geq \delta_{z}.$ By Lemma \ref{lem:lower-bound}, it is equivalent to
\begin{equation*}
\left[-\delta_{z}, \delta_{z}\right] \cap \supp (\mu_1 \boxplus \mu_2) = \emptyset.
\end{equation*}
By our assumption, we have $\vert a_n -z\vert \geq \delta_{z}$ for all large $n.$ Let
$\mu_1^{(n)} = \tilde{\mu}_{|a_n -z|},$ we have $\mu_1^{(n)} \rightarrow \mu_1$ weakly as $n \rightarrow \infty,$ and 
\begin{equation*}
\left[-\delta_{z}, \delta_{z} \right] \cap \supp (\mu_1^{(n)}) = \emptyset
\end{equation*}
for all large $n.$ 
Therefore, Proposition \ref{prop:no-outlier-1} follows by Theorem \ref{thm:support}. 

For any $z \in \Theta_{\rm in},$ By Lemma \ref{lem:lower-bound}, there exists $\delta'_{z} >0$ such that 
\begin{equation*}
\left[-\delta'_{z}, \delta'_{z}\right] \cap \supp (\mu_1 \boxplus \mu_2) = \emptyset.
\end{equation*}
On the other hand, by Corollary \ref{cor:no-atom-omega_1}, we may also assume that 
\begin{equation*}
\left[-\delta'_{z}, \delta'_{z} \right] \cap \supp (\mu_2) = \emptyset.
\end{equation*}
Therefore, Corollary \ref{cor:no-outlier-1} follows by Theorem \ref{thm:support-in}.
\end{proof}


\subsection{Proof of Proposition \ref{prop:no-outlier-2}}

Denote $A_n^z:= A_n -z$ and $a_n^z:= a_n-z.$ We write
\begin{equation*}
\begin{split}
(M_n-z) (M_n-z)^* &= (A_n^z+ U_n\Sigma_n) (A_n^z + U_n\Sigma_n)^*\\
:&= Q (U_n, U_n^*, A_n^z, (A_n^z)^*, \Sigma_n),
\end{split}
\end{equation*}
where $Q= Q(x_1, x_1^*, x_2, x_2^*, x_3)$ is a self-adjoint polynomial and $x_3$ is a tuple of self-adjoint indeterminates.
Suppose $T$ has the same $*$-distribution as $u\Sigma$, where $u$ is a Haar unitary, $\Sigma\geq 0$, and $u, \Sigma$ are freely independent. Hence, it is sufficient to show that for any $\varepsilon>0,$ almost surely for all large $n$, we have 
\begin{equation}\label{eq:spectrum}
{\rm spec} (Q (U_n, U_n^*,  A_n^z, (A_n^z)^*, \Sigma_n)) \subseteq {\rm spec} (Q (u, u^*, a_n^z, (a_n^z)^*, \Sigma))  + [-\varepsilon, \varepsilon]. 
\end{equation}
We merge the approaches in Collins-Male's work \cite{CM2014} on the strong asymptotic freeness of Haar unitary matrices and Belinschi-Capitaine's work \cite{Serban2017JFA} on the strong convergence of polynomials of GUE matrices.

\begin{definition}
Let $x\in\mathcal{A}$ be a self-adjoint element. Denote $F_x$ by the cumulative function of $\mu_x$ defined by 
\begin{equation}
F_x (t) : = \mu_x (( -\infty, t] ), \qquad t\in\mathbb{R}. 
\end{equation}
The generalized inverse of $F_x$ is given by
\begin{equation}
F^{-1}_x (s)= \inf \{t \in [-\|x\|, \|x\|] : F_x (t) \geq s \}
\end{equation}
for any $s \in  (0,1].$
\end{definition}

Define the matrix 
\begin{equation}
D_n: = V_n {\rm diag} \left( \frac{1}{n}, \ldots, \frac{n-1}{n}, \frac{n}{n}\right) V_n^*,
\end{equation}
where $V_n$ is a unitary Haar matrix, independent of $A_n.$ Let $d$ be the strong limit of $D_n$ in $(\mathcal{A}, \tau),$ whose spectral distribution is the uniform measure on $[0, 1].$
Let $G_n$ be a GUE matrix independent from $A_n,$ and $g$ be a semi-circular element free from $a_n$ in $(\mathcal{A}, \tau).$ Note that $G_n$ and $D_n$ are considered as self-adjoint elements in $(M_n (\mathbb{C}), {\rm tr}_n).$
Then we have (see \cite[Lemma 3.1]{CM2014})
\begin{equation}
F_{G_n} (G_n) = D_n \; \text{and} \; F_g (g)=d.
\end{equation}
Moreover, by Dini's theorem, $F_{G_n}$ convergence uniformly to $F_g.$

Now we are ready to prove Proposition \ref{prop:no-outlier-2}. We first need the following technical lemma, which was essentially given in \cite{BH2022}.

\begin{lemma}\label{lem:dist-ineq}
For any $\varepsilon >0$, let $X_n$ be a self-adjoint matrix in $M_n (\mathbb{C})$ and $x$ be a self-adjoint operator in $\mathcal{A}$ satisfy 
\begin{equation}
\|(\lambda - X_n)^{-1}\|  \leq \|(\lambda - x)^{-1}\| + \frac{\varepsilon} {({\rm Im} \lambda)^2} 
\end{equation}
for all $\lambda = {\rm Re} \lambda  + i \varepsilon/2$ with ${\rm Re} \lambda \in {\rm spec} (X_n).$
Then 
\begin{equation}
{\rm spec} (X_n) \subseteq {\rm spec} (x) + \varepsilon [-1, 1].
\end{equation}
\end{lemma}

\begin{proof}
We adapt the proof of \cite[Lemma 7.1]{BH2022} to our case.  
Note that 
\begin{equation}
	 \label{lem:dist}
\|(\lambda - X_n)^{-1}\| =  \frac{1}{{\rm dist} (\lambda, {\rm spec}(X_n))} \;\text{and}  \; \|(\lambda - x)^{-1}\| =  \frac{1}{{\rm dist} (\lambda, {\rm spec}(x))}. 
\end{equation}
So the assumption implies that 
\begin{equation*}
\frac{2}{\varepsilon} \leq \frac{1}{\sqrt{\varepsilon^2/4 + {\rm dist} ( {\rm Re} \lambda , {\rm spec}(x))^2}} +  \frac{\varepsilon} {\varepsilon^2}. 
\end{equation*}
If ${\rm dist} ( {\rm Re} \lambda , {\rm spec}(x))> \varepsilon,$ then $\frac{1}{ \varepsilon} <  \frac{\varepsilon} {\varepsilon^2}$, which is a contradiction. Thus, for any
${\rm Re} \lambda \in {\rm spec}(X_n),$
we have ${\rm dist} ( {\rm Re} \lambda , {\rm spec}(x))\leq \varepsilon.$ 
\end{proof}

\begin{proof}[Proof of Proposition \ref{prop:no-outlier-2}]
For simplicity, we denote
\begin{equation}
Q (x_1, x_2, x_3) := Q(x_1, x_1^*, x_2, x_2^*, x_3).
\end{equation}

\noindent--{\it Step 1}. We claim that for any $ \varepsilon>0,$ almost surely for all large $n$, 
\begin{equation}\label{eq:M_n}
{\rm spec} (Q (D_n, A_n^z, \Sigma_n )) \subseteq {\rm spec} (Q (d, a_n^z, \Sigma))  + [-\varepsilon, \varepsilon],
\end{equation}
where $d, a_n, T$ are freely independent in $(\mathcal{A}, \tau).$ 

For any $\lambda \in \mathbb{C}^+,$ we consider
\begin{equation}\label{eq:I+II}
\begin{split}
& \left| \| (\lambda- Q(d, a_n^z, \Sigma))^{-1}\| -  \| (\lambda- Q(D_n, A_n^z, \Sigma_n))^{-1}\|\right| \\
&  \quad \quad \leq \left| \| (\lambda- Q(F_g (g), a_n^z, \Sigma))^{-1}\| -  \| (\lambda- Q(F_g (G_n), A_n^z, \Sigma_n))^{-1}\|\right| \\
& \quad \quad  +  \left| \| (\lambda- Q(F_g(G_n), A_n^z, \Sigma_n))^{-1} \| - \| (\lambda- Q( F_{G_n} (G_n), A_n^z, \Sigma_n))^{-1}\| \right| \\
&  \quad \quad \leq \left| \| (\lambda- Q(F_g (g), a_n^z, \Sigma))^{-1}\| -  \| (\lambda- Q(F_g (G_n), A_n^z, \Sigma_n))^{-1}\|\right| \\
& \quad \quad  +  \| (\lambda- Q(F_g(G_n), A_n^z, \Sigma_n))^{-1}- (\lambda- Q( F_{G_n} (G_n), A_n^z, \Sigma_n))^{-1}\|\\
&  \quad \quad =: (I)+(II).
\end{split}
\end{equation}
For the term $(I)$, by \eqref{lem:dist}, 
\begin{equation*}
\begin{split}
(I) & =  \left| \| (\lambda- Q(F_g (g), a_n^z, \Sigma))^{-1}\| -  \| (\lambda- Q(F_g (G_n), A_n^z, \Sigma_n))^{-1}\|\right| \\
&  = \left|  \frac{1}{{\rm dist} (\lambda, {\rm spec} (Q (F_g (g), a_n^z, \Sigma)))} - \frac{1}{{\rm dist} (\lambda, {\rm spec} (Q (F_g (G_n), A_n^z, \Sigma_n)))} \right| \\
&  \leq \frac{1} { ({\rm Im} \lambda)^2} \cdot {\rm dist} ({\rm spec} (Q (F_g (g), a_n^z, \Sigma)), {\rm spec} (Q (F_g (G_n), A_n^z, \Sigma_n)) ).
\end{split}
\end{equation*}
Recall that $g$ and $\Sigma$ are the strong limit of $G_n$ and $\Sigma_n,$ respectively.  And  $a_n$ has the same $*$-distribution of $A_n$. By \cite[Theorem 1.1]{Serban2017JFA}, we have 
\begin{equation*}
\lim_{n \rightarrow \infty} {\rm dist} ({\rm spec} (Q (F_g (g), a_n^z, \Sigma)), {\rm spec} (Q (F_g (G_n), A_n^z, \Sigma_n)) ) =0.
\end{equation*}
For the term $(II)$, by the resolvent identity, we have 
\begin{equation*}
\begin{split}
(II) & =\| (\lambda- Q(F_g(G_n), A_n^z, \Sigma_n))^{-1}- (\lambda- Q( F_{G_n} (G_n), A_n^z, \Sigma_n))^{-1}\| \\
& \leq \frac{1}{ ({\rm Im} \lambda)^2} \|Q(F_g(G_n), A_n^z, \Sigma_n)- Q( F_{G_n} (G_n), A_n^z, \Sigma_n) \|.
\end{split}
\end{equation*}
Since $F_{G_n}$ converges uniformly to $F_g,$ and $G_n, A_n^z, \Sigma_n$ are uniformly bounded in operator norm, we have
\begin{equation*}
\lim_{n \rightarrow \infty} \|Q(F_g(G_n), A_n^z, \Sigma_n)- Q( F_{G_n} (G_n), A_n^z, \Sigma_n) \| =0.
\end{equation*}

Therefore, for any $\varepsilon>0,$ there exists $n_0,$ such that for all $n \geq n_0,$
\begin{equation*}
 \left| \| (\lambda- Q(d, a_n^z, \Sigma))^{-1}\| -  \| (\lambda- Q(D_n, A_n^z, \Sigma_n))^{-1}\|\right| \leq \frac{\varepsilon}{({\rm Im} \lambda)^2}.
\end{equation*}
Then by Lemma \ref{lem:dist-ineq}, we deduce \eqref{eq:M_n}. 

\vspace{3mm}

\noindent  --{\it Step 2}.  We write $U_n = V_n {\rm diag} (e^{2\pi i \theta_1^{(n)}}, \ldots, e^{2\pi i \theta_n^{(n)}}) V_n^*,$ where $\theta_i ^{(n)}\in [0, 2\pi), i =1, \ldots, n$ and $\theta_1^{(n)} \leq \theta_2^{(n)} \leq \cdots \leq \theta_n^{(n)}.$ Denote by $F_{U_n}$ the cumulative function of ${\rm diag} ( \theta_1^{(n)}, \theta_2^{(n)}, \ldots, \theta_n^{(n)}).$ 
Denote $\gamma_n : t \rightarrow \exp (2 \pi i  F_{U_n}^{-1} (t))$ and $\gamma: t \rightarrow \exp (2 \pi i t).$ By \cite[Lemma 3.3]{CM2014}, we have 
\begin{equation}
U_n = \gamma_n (D_n) \; \text{and} \; u = \gamma(d),
\end{equation}
and $\gamma_n$ converges uniformly to $\gamma$ almost surely. 

Similar to \eqref{eq:I+II} in {\it Step 1}, we have 
\begin{equation*}
\begin{split}
& \left| \| (\lambda- Q(u, a_n^z, \Sigma))^{-1}\| -  \| (\lambda- Q(U_n, A_n^z, \Sigma_n))^{-1}\|\right| \\
& \quad \leq \left| \| (\lambda- Q( \gamma (d), a_n^z, \Sigma))^{-1}\| -  \| (\lambda- Q(\gamma (D_n),  A_n^z, \Sigma_n))^{-1}\|\right| \\
& \quad  +  \| (\lambda- Q(\gamma(D_n),  A_n^z, \Sigma_n))^{-1}- (\lambda- Q( \gamma_n (D_n),  A_n^z,  \Sigma_n))^{-1}\|\\
&  \quad =: (III) + (IV).
\end{split}
\end{equation*}
For the term $(III)$, similar to {\it Step 1}, 
\begin{equation*}
\begin{split}
 \Big| \| (\lambda & - Q( \gamma (d),  a_n^z, \Sigma))^{-1}\| -  \| (\lambda- Q(\gamma (D_n),  A_n^z, \Sigma_n))^{-1}\| \Big|\\
&  \leq \frac{1} { ({\rm Im} \lambda)^2} \cdot {\rm dist} ({\rm spec} (Q (\gamma (d), a_n^z, \Sigma)), {\rm spec} (Q (\gamma (D_n), A_n^z, \Sigma_n)) ).
\end{split}
\end{equation*}
For any $\varepsilon>0$, by {\it Step 1} and the continuity of $\gamma,$ almost surely for all large $n$, we have 
\begin{equation*}
{\rm dist} ({\rm spec} (Q (\gamma (d), a_n^z, \Sigma)), {\rm spec} (Q (\gamma (D_n), A_n^z, \Sigma_n)) ) \leq \varepsilon.
\end{equation*}
For the term $(IV)$, 
\begin{equation*}
\begin{split}
\| (\lambda &- Q(\gamma(D_n),  A_n^z, \Sigma_n))^{-1}- (\lambda- Q( \gamma_n (D_n), A_n^z, \Sigma_n))^{-1}\|\\
 &\leq \frac{1}{ ({\rm Im} \lambda)^2} \|Q(\gamma (D_n), A_n^z, \Sigma_n)- Q( \gamma_n (D_n), A_n^z, \Sigma_n) \| .
\end{split}
\end{equation*}
Since $\gamma_n$ converges uniformly to $\gamma,$ and $D_n, A_n^z, T_n$ are uniformly bounded in operator norm, we have
\begin{equation*}
\lim_{n \rightarrow \infty} \|Q(\gamma (D_n), A_n^z, \Sigma_n)- Q( \gamma_n (D_n), A_n^z, \Sigma_n) \| =0.
\end{equation*}

Therefore, for any $\varepsilon>0,$ there exists $n_0,$ such that for all $n \geq n_0,$
\begin{equation*}
 \left| \| (\lambda- Q(u, a_n^z, \Sigma))^{-1}\| -  \| (\lambda- Q(U_n, A_n^z, \Sigma_n))^{-1}\|\right| \leq \frac{\varepsilon}{({\rm Im} \lambda)^2}.
\end{equation*}
Then by Lemma \ref{lem:dist-ineq}, for any $ \varepsilon>0,$ almost surely for all large $n$, 
\begin{equation*}
{\rm spec} (Q (U_n, A_n^z, \Sigma_n)) \subseteq {\rm spec} (Q (u, a_n^z, \Sigma))  + [-\varepsilon, \varepsilon],
\end{equation*}
which completes our proof.
\end{proof}


\section{Outliers in the outer domain; proof of Theorem \ref{thm:outlier}}\label{sec:stable-outlier}

Recall the resolvent matrices
\begin{equation}
R_n (z) = (Y_n + A_n^{'} -z)^{-1} \;\; \text{and} \; R_n^{'}(z) = (A_n^{'}-z)^{-1}.
\end{equation}

\subsection{Spectral radius of $R_n^{'}(z) Y_n$}

Recall that $\mu_z$ is the distribution of 
\begin{equation*}
(a+ T -z) (a +T -z)^*.
\end{equation*}
Define for any $\omega, z \in \mathbb{C},$ $\mu_{\omega, z}$ as the distribution of 
\begin{equation*}
(a+ \omega T  -z) (a+ \omega T-z)^*.
\end{equation*}
Then we have the following lemma.

\begin{lemma}\label{lem:mu_{omega, z}}
Let $\Gamma$ be a compact subset in $\Theta_{\rm out}.$  Then there exists $\rho >1$ such that for any $\omega \in \mathbb{C}$ such that 
$|\omega| \leq \rho$ and any $z \in \Gamma,$ we have $0 \notin {\rm supp} (\mu_{\omega, z}).$
\end{lemma}

\begin{proof}
Let $\mu_1 =\tilde{\mu}_{|a-z|}$ and $\mu_2= \tilde\mu_{|T|},$ and denote $\mu_{2, \omega}: = \tilde{\mu}_{|\omega T|};$ see Subsection \ref{subsec:sum-R-diagonal}. 
Hence, Lemma \ref{lem:mu_{omega, z}-0} implies that $0 \notin \supp (\mu_1\boxplus \mu_{2, \omega}).$ By Lemma \ref{lem:lower-bound}, $|a+ \omega T  -z|$ is invertible, and thus 
$0 \notin {\rm supp} (\mu_{\omega, z}).$
\end{proof}

\begin{proposition}\label{prop:spectral-radius}
Let $\Gamma$ be a compact subset in $\Theta_{\rm out}.$ Assume that Assumption \ref{assump:Ab} holds and Assumptions \ref{assump:Xa}--\ref{assump:Xb} hold for $\Gamma$. There exists $0 < \varepsilon_0 <1$ such that almost surely for large $n$, we have, 
\begin{equation*}
\sup_{z \in \Gamma} r (R_n^{'}(z) Y_n) \leq 1- \varepsilon_0,
\end{equation*}
where $r (X_n)$ is the spectral radius of a matrix $X_n.$
\end{proposition}

\begin{proof}
We adapt the argument from \cite[Lemma 4.5]{BordenaveC_cpam2016_outlier} and \cite[Lemma 6.5]{Serban2021}.
Let $\tilde{\Gamma} = \{(\omega, z) \in \mathbb{C}^2: |\omega| \leq \rho, z \in \Gamma\}$ where $\rho$ is given in Lemma \ref{lem:mu_{omega, z}}. Hence, $0 \notin {\rm supp} (\mu_{\omega, z})$ for any 
$(\omega, z) \in \tilde{\Gamma}.$ By Proposition \ref{prop:no-outilier}, there exists $\gamma_{\omega, z} >0$ such that almost surely for all large $n,$ there is no singular value of 
\begin{equation*}
\omega Y_n + A_n^{'}- z 
\end{equation*}
in $[0, \gamma_{\omega, z}].$ Moreover, we have $\|Y_n\| \leq C_{\ref{assump:Aa}}.$ Then, by a compactness argument, we can prove that there exists $\eta >0$ such that almost surely for all large $n,$ there is no singular value of 
$\omega Y_n +A_n^{'}- z$ in $[0, \eta]$ for any $(\omega, z) \in \tilde{\Gamma}.$

Assume that $\lambda \neq 0$ is an eigenvalue of $R_n^{'}(z) Y_n.$ Then, $z$ must be an eigenvalue of $-\lambda^{-1}Y_n + A_n^{'}$; see \cite[Lemma 4.5]{BordenaveC_cpam2016_outlier}, i.e., $0$ must be a singular value of $-\lambda^{-1}Y_n + A_n^{'}-z.$ Therefore, by the preceding paragraph, we must have $1/|\lambda| >\rho>1,$ which completes the proof.
\end{proof}

\begin{proposition}\label{prop:norm-resolvent}
Let $\Gamma$ be a compact subset in $\Theta_{\rm out}.$ Assume that Assumption \ref{assump:Ab} holds and Assumptions \ref{assump:Xa}--\ref{assump:Xb} hold for $\Gamma$. There exists $0 < \varepsilon_0 <1$ and $C_{\ref{prop:norm-resolvent}} >0$ such that almost surely for large $n$, for any $k \geq 1$, we have 
\begin{equation*}
\sup_{z \in \Gamma} \Vert (R_n^{'}(z) Y_n)^k \Vert \leq C_{\ref{prop:norm-resolvent}} (1- \varepsilon_0)^k.
\end{equation*}
\end{proposition}

\begin{proof}
The proof is similar to the proofs for \cite[Proposition 4.3]{BordenaveC_cpam2016_outlier} and \cite[Proposition 6.6]{Serban2021}. For $z \in \Gamma,$ set $T_n (z)= R_n^{'}(z) Y_n.$ 
Choose $0< \varepsilon < \min\{\varepsilon_0, 1-1/\rho\},$ where $\varepsilon_0$ and $\rho$ are given in Proposition \ref{prop:spectral-radius} and Lemma \ref{lem:mu_{omega, z}}, respectively. 
By Proposition \ref{prop:spectral-radius} and Dunford-Riesz calculus, we have almost surely for all large $n,$
\begin{equation*}
(T_n(z))^k = \frac{1}{2\pi i} \int_{|\omega| =1-\varepsilon} \omega^k (\omega - T_n(z))^{-1} {\rm d} \omega, \;\; k \geq 0.
\end{equation*}
Hence,
\begin{equation*}
\Vert (T_n(z))^k  \Vert \leq  (1-\varepsilon)^{k+1} \cdot \sup_{|\omega| =1-\varepsilon} \Vert (\omega - T_n(z))^{-1} \Vert, \;\; k \geq 0.
\end{equation*}

For any $\omega$ such that $|\omega| = 1- \varepsilon,$ we have 
\begin{equation*}
\omega -T_n (z) = \omega R_n^{'}(z) \left( -\frac{1}{\omega} Y_n + A_n^{'}-z \right).
\end{equation*}
Since $1/|\omega| < \rho,$ there exists $\eta >0$ such that there is no singular value of $-\frac{1}{\omega} Y_n + A_n^{'}-z$  in $[0, \eta];$ see the proof of Proposition \ref{prop:spectral-radius}. 
It follows that $\Vert -\frac{1}{\omega} Y_n + A_n^{'}-z  \Vert \leq \frac{1}{\eta}.$ Therefore, there exists $C>0$, such that for any $z \in \Gamma,$
\begin{equation*}
 \Vert (\omega - T_n(z))^{-1} \Vert  \leq \frac{1}{|\omega|} \left\Vert -\frac{1}{\omega} Y_n + A_n^{'}-z  \right\Vert  \cdot \Vert A_n^{'} -z\Vert \leq C,
\end{equation*}
which completes our proof.
\end{proof}


\subsection{Convergence of a random matrix polynomial}
We denote by $\mathbb{C}\langle x_1, \ldots, x_k \rangle$ the set of noncommutative polynomials in the noncommutative variables $\{x_1, \ldots, x_k\}.$  We have the following result replacing \cite[Proposition 4.1]{BordenaveC_cpam2016_outlier} in our setting.
\begin{proposition}\label{prop:isotropic}
Let $P \in \mathbb{C}\langle x_1, \ldots, x_k \rangle$ such that the total exponent of $x_k$ in each monomial of $P$ is nonzero. We consider a
sequence $(B^{(1)}_n, \ldots, B^{(k-1)}_n)\in M_n(\mathbb{C})^{k-1}$ of matrices with operator norm uniformly bounded in $n$ and unit vectors $c, b$ in $\mathbb{C}^n$. Then,  
\begin{equation*}
c^* P \left(B^{(1)}_n, \ldots, B^{(k-1)}_n, Y_n \right) b  \rightarrow 0 \; \text{in probability}
\end{equation*}
as $n$ goes to infinity. 
\end{proposition}

Our main tool is the Weingarten calculus that was initially introduced by Weingarten in physics literature \cite{Weingarten-78} and developed further by Collins \cite{Collins-IMRN03} and Collins-\'{S}niady \cite{Collins-CMP06}. Let $\mathcal{S}_p$ be the permutation group with $p$ elements. For a given permutation $\sigma\in \mathcal{S}_p$, we denote by $\#\sigma$  the number of cycles in $\sigma$ and  by $\vert\sigma\vert$ the minimum number of transpositions that multiply to $\sigma$. It is well-known that (see \cite[Lecture 23]{SpeicherNicaBook})
\begin{equation*}
	\vert\sigma\vert=p-\#\sigma.
\end{equation*}

\begin{definition}[\cite{Collins-IMRN03, Collins-CMP06}]
For the $n \times n$ unitary matrix $U_n,$ the Weingarten function $\operatorname{Wg}(\sigma)$ is given by the inverse of the function $\sigma\mapsto n^{\#\sigma}$ with respect to the following convolution formula 
	$$\sum\limits_{\tau\in\mathcal{S}_p}\operatorname{Wg}(\sigma\tau^{-1})n^{\#(\tau\pi^{-1})}=\delta_{\sigma,\pi}$$
	for all $\sigma,\pi\in\mathcal{S}_p$.
\end{definition}

\begin{proposition}[\cite{Collins-IMRN03, Collins-CMP06}] 
\label{prop:W-calculus}
Let $U_n =(U_{ij})_{i,j=1}^n$ be a Haar-distributed unitary matrix. Let $(i_1,\ldots,i_p)$, $(i_1^{\prime},\ldots,i_p^{\prime})$, $(j_1,\ldots,j_p)$, and $(j_1^{\prime},\ldots,j_p^{\prime})$ be $p$-tuples of positive integers from $\{1,2,\ldots,n\}$. Then 
	\begin{equation}\label{eq:Weingarten}
	\begin{split}
		\mathbb{E} & \left[  U_{i_1j_1}\cdots U_{i_pj_p}\overline{U}_{i^{\prime}_1j^{\prime}_1} \cdots \overline{U}_{i^{\prime}_pj^{\prime}_p} \right]\\
		& =\sum\limits_{\alpha, \beta\in\mathcal S_p}\delta_{i_1i^{\prime}_{\alpha(1)}}\cdots\delta_{i_pi^{\prime}_{\alpha(p)}}\delta_{j_1j^{\prime}_{\beta(1)}}\cdots\delta_{j_p j^{\prime}_{\beta(p)}}\operatorname{Wg}(\alpha\beta^{-1}).
		\end{split}
	\end{equation}
	If $p\neq p^{\prime}$, then 
	\begin{equation*}
		\mathbb{E} \left[ U_{i_1j_1}\cdots U_{i_pj_p}\overline{U}_{i^{\prime}_1j^{\prime}_1}\cdots\overline{U}_{i^{\prime}_pj^{\prime}_p} \right]=0.
	\end{equation*}
\end{proposition}

We will also use the following asymptotic formula \cite{Collins-CMP06} for the Weingarten function.
\begin{equation}\label{eq:asymptotic-W}
	\operatorname{Wg}(\sigma)=n^{-(p+\vert\sigma\vert)}(\operatorname{Mob}(\sigma)+O(n^{-2})),
\end{equation}
where $\operatorname{Mob}(\sigma)$ is the M\"{o}bius function on the lattice of partitions. The precise definition of $\operatorname{Mob}(\sigma)$ is not important here. The reader is referred to  \cite[Lecture 10]{SpeicherNicaBook} for details. 

\begin{proof}[Proof of Proposition \ref{prop:isotropic}]
It is sufficient to prove Proposition \ref{prop:isotropic} for $P$ of the form
\begin{equation}
P\left({\bf B}_n, Y_n \right) = Y_n \prod_{l =1}^{k-1} \left( B^{(l)}_n Y_n \right),
\end{equation}
where ${\bf B}_n = (B^{(1)}_n, \ldots, B^{(k-1)}_n)$ and for any $1 \leq l \leq k-1,$ $\Vert B^{(l)}_n \Vert \leq 1;$ see \cite[Proposition 4.1]{BordenaveC_cpam2016_outlier}.

Recall that $Y_n = U_n\Sigma_n,$ where $U_n = (U_{ij})_{i,j=1}^n$ and $\Sigma_n = {\rm diag} (s_1, \ldots, s_n).$ Let $c= (c_1, \ldots, c_n),$ $b= (b_1, \ldots, b_n)$, $Y_n = (Y_{ij})_{i,j=1}^n$, and $B^{(l)}_n = (B^{(l)}_{ij})_{i,j=1}^n$, then we have 
\begin{equation*}
\begin{split}
c^* P\left({\bf B}_n, Y_n \right) b  &= \sum_{i_1, \ldots, i_{2k}=1}^n \bar{c}_{i_1} b_{i_{2k}} \cdot  Y_{i_1i_2} \prod_{l=1}^{k-1} (B_{i_{2l} i_{2l+1}}^{(l)} Y_{i_{2l+1} i_{2l+2}}) \\
& =  \sum_{i_1, \ldots, i_{2k}=1}^n \bar{c}_{i_1} b_{i_{2k}} \cdot \left(\prod_{l=1}^{k} s_{i_{2l}} \right)  \cdot \left(\prod_{l=1}^{k-1} B_{i_{2l} i_{2l+1}}^{(l)} \right)  \cdot  U_{i_1i_2} \cdots  U_{i_{2k-1} i_{2k}},
\end{split}
\end{equation*}
and so we have 
\begin{equation*}
\begin{split}
\mathbb{E} & \left[ \left| c^* P\left( {\bf B}_n, Y_n \right) b \right|^2\right]\\
& = \sum_{\substack {i_1, \ldots, i_{2k}\\ j_1, \ldots, j_{2k}}} \bar{c}_{i_1} c_{j_1}b_{i_{2k}}  \bar{b}_{j_{2k}}  \cdot \left(\prod_{l=1}^{k} s_{i_{2l}}  s_{j_{2l}}\right)  \cdot \left(\prod_{l=1}^{k-1} B_{i_{2l} i_{2l+1}}^{(l)} \bar{B}_{j_{2l} j_{2l+1}}^{(l)} \right) \\
&  \quad \quad \cdot   \mathbb{E} \left [ U_{i_1i_2} \cdots  U_{i_{2k-1} i_{2k}} \bar{U}_{j_1j_2} \cdots  \bar{U}_{j_{2k-1} j_{2k}} \right].
\end{split}
\end{equation*}

The Weingarten formula \eqref{eq:Weingarten} implies that 
\begin{equation*}
\begin{split}
&\mathbb{E} \left [ U_{i_1i_2} \cdots  U_{i_{2k-1} i_{2k}} \bar{U}_{j_1j_2} \cdots  \bar{U}_{j_{2k-1} j_{2k}} \right] \\
& = \sum_{\sigma, \tau \in S_k} \delta_{i_1, j_{\sigma(1)}} \delta_{i_3, j_{\sigma(3)}} 
\cdots \delta_{i_{2k-1}, j_{\sigma(2k-1)}} \cdot \delta_{i_2, j_{\tau(2)}} \delta_{i_4, j_{\tau(4)}} 
\cdots \delta_{i_{2k}, j_{\tau(2k)}} \cdot {\rm Wg} (\tau \sigma^{-1}).
\end{split}
\end{equation*}
Hence, 
\begin{equation*}
\begin{split}
\mathbb{E} & \left[ \left| c^* P\left( {\bf B}_n, Y_n \right) b \right|^2\right] \\
&= \sum_{ j_1, \ldots, j_{2k}=1}^n  \sum_{\sigma, \tau \in S_k} \bar{c}_{j_{\sigma(1)}} c_{j_1}b_{j_{\tau(2k)}}  \bar{b}_{j_{2k}} \\
& \quad \quad \cdot \left(\prod_{l=1}^{k} s_{j_{\tau(2l)}}  s_{j_{2l}}\right) \cdot \left(\prod_{l=1}^{k-1} B_{j_{\tau(2l)} j_{\sigma(2l+1)}}^{(l)} \bar{B}_{j_{2l} j_{2l+1}}^{(l)} \right) \cdot {\rm Wg} (\tau \sigma^{-1})
\end{split}
\end{equation*}
By using the 
asymptotic formula for the Weingarten function \eqref{eq:asymptotic-W}, we obtain
\begin{equation*}
{\rm Wg} (\sigma) = n^{-(k+ |\sigma|)} \left( {\rm Mob} (\sigma) + O(n^{-2})\right), \; \text{for} \; \sigma \in S_k.
\end{equation*}
It then follows that
\begin{equation*}
\begin{split}
\mathbb{E} & \left[ \left| c^* P\left( {\bf B}_n, Y_n \right) b \right|^2\right] \\
& \leq \sum_{ j_1, \ldots, j_{2k}=1}^n  \sum_{\sigma, \tau \in S_k} \frac{1}{2} \left( |\bar{c}_{j_{\sigma(1)}}|^2 + |c_{j_1} |^2 \right) \cdot \frac{1}{2} \left( |b_{j_{\tau(2k)}} s_{j_{\tau(2k)}}|^2 +  |\bar{b}_{j_{2k}} s_{j_{2k}}|^2 \right)  \\
& \quad \quad  \cdot \frac{1}{2} \left\{ \left(\prod_{l=1}^{k-1} s_{j_{2l}}^2 |\bar{B}_{j_{2l} j_{2l+1}}^{(l)} |^2 \right)  + \left(\prod_{l=1}^{k-1}s_{j_{\tau(2l)}}^2  |B_{j_{\tau(2l)} j_{\sigma(2l+1)}}^{(l)}|^2 \right)  \right\} \cdot O \left( n^{-k}\right)\\
& \leq (k!)^2 \sum_{ j_1, \ldots, j_{2k}=1}^n |c_{j_1}|^2 \cdot |b_{j_{2k}}|^2 \cdot s_{j_{2k}}^2 \cdot \left(\prod_{l=1}^{k-1} s_{j_{2l}}^2 |B_{j_{2l} j_{2l+1}}^{(l)} |^2 \right) \cdot O \left( n^{-k}\right)\\
& = (k!)^2 \left( \sum_{j_{2k}=1}^n  |b_{j_{2k}}|^2 \cdot s_{j_{2k}}^2 \right) \cdot \left(\sum_{j_2, \ldots, j_{2k-1}=1}^n \prod_{l=1}^{k-1} s_{j_{2l}}^2 |B_{j_{2l} j_{2l+1}}^{(l)} |^2 \right) \cdot O \left( n^{-k}\right).
\end{split}
\end{equation*}

Note that
\begin{equation*}
\sum_{j_{2k}=1}^n  |b_{j_{2k}}|^2 \cdot s_{j_{2k}}^2 \leq \|\Sigma_n\|^2 \cdot \sum_{j_{2k}=1}^n  |b_{j_{2k}}|^2 \leq C_{\ref{assump:Ab}}^2, 
\end{equation*}
and 
\begin{equation*}
\begin{split}
&\frac{1}{n^{k-1}} \sum_{j_2, \ldots, j_{2k-1}=1}^n \prod_{l=1}^{k-1} s_{j_{2l}}^2 |B_{j_{2l} j_{2l+1}}^{(l)} |^2 \\
&= \frac{1}{n^{k-1}}\sum_{j_2, j_4\ldots, j_{2k-2}=1}^n s_{j_2}^2 s_{j_4}^2 \cdots s_{j_{2k-2}}^2 \cdot \left( \sum_{j_3} |B_{j_{2} j_{3}}^{(1)} |^2\sum_{j_5} |B_{j_{4} j_{5}}^{(2)} |^2  \cdots \sum_{j_{2k-1}} |B_{j_{2k-2} j_{2k-1}}^{(k-1)} |^2\right)\\
& \leq \left( \frac{1}{n}\sum_{j=1}^n s_j^2  \right)^{k-1}.
\end{split}
\end{equation*}
It follows that
\begin{equation*}
\begin{split}
\mathbb{E}  \left[ \left| c^* P\left({\bf B}_n, Y_n \right) b \right|^2\right]  \leq C_{\ref{assump:Ab}}^2  (k!)^2   \left( \frac{1}{n}\sum_{j=1}^n s_j^2  \right)^{k-1} \cdot O \left( n^{-1}\right) =  O \left( n^{-1} \right).
\end{split}
\end{equation*}

For any $\varepsilon>0$, the Markov inequality implies that
\begin{equation*}
\mathbb{P} \left( \left| c^* P\left({\bf B}_n, Y_n \right) b \right|^2 \geq \varepsilon \right) \leq  \mathbb{E}  \left[ \left| c^* P\left({\bf B}_n, Y_n \right) b \right|^2\right]/\varepsilon \leq O \left( n^{-1} \right)/ \varepsilon. 
\end{equation*}
So 
\begin{equation*}
\lim_{n \to \infty} \mathbb{P} \left( \left| c^* P\left({\bf B}_n, Y_n \right) b \right|^2  \geq \varepsilon \right) =0
\end{equation*}
for any $\varepsilon>0$, which implies our result. 
\end{proof}

\subsection{Convergence of resolvent in the outer domain}
Let $\Gamma$ be as in Theorem \ref{thm:outlier}. From the singular value decomposition of $A_n^{''},$ we can write $A_n^{''} = P_n Q_n$ with $P_n, Q_n^T \in M_{n, r}(\mathbb{C})$ with uniformly bounded norms. Recall
\begin{equation*}
R_n(z) = \left(A_n^{'} + Y_n -z \right)^{-1} \; \text{and} \; R_n^{'}(z) = (A_n^{'} -z)^{-1}.
\end{equation*}

By using the resolvent identity, we have
\begin{equation*}
R_n (z) = R_n^{'}(z) - R_n^{'}(z) Y_n R_n(z).
\end{equation*}
Hence by an induction argument, for any integer number $K \geq 2$ (see \cite[Equation 6.5]{Serban2021}),
\begin{equation}\label{eq:induction-R_n}
\begin{split}
Q_nR_n (z)P_n & - Q_nR_n^{'}(z)P_n \\
& = \sum_{k=1}^{K-1} (-1)^k Q_n (R_n^{'}(z) Y_n)^k R_n^{'}(z) P_n + (-1)^K Q_n (R_n^{'}(z) Y_n)^K R_n (z)P_n.
\end{split}
\end{equation}

\begin{lemma}\label{lem:series-R_n}
Suppose that Assumption \ref{assump:Ab} holds and Assumptions \ref{assump:Xa}--\ref{assump:Xb} hold with $\Gamma \subseteq \Theta_{\rm out}$ compact. Then for any
$z \in \Gamma,$ 
\begin{equation*}
\left \Vert Q_n \sum_{k=1}^\infty (-1)^k (R_n^{'}(z)Y_n)^k R_n^{'}(z) P_n  \right\Vert \rightarrow 0 \; \text{in probability}
\end{equation*}
as $n$ goes to infinity.
\end{lemma}

\begin{proof}
Let $A_n^{''} = \sum_{i,j =1}^r s_i v_i u_j^*$ be the singular value decomposition of $A_n^{''},$ where $v_i$ and $u_j$ are unit column vectors and $s_i$ is a singular value of $A_n^{''}.$
For any $i, j =1, \ldots, r,$
\begin{equation*}
(Q_n (R_n^{'}(z)Y_n)^k R_n^{'}(z) P_n)_{i, j} = s_i v_i^* (R_n^{'}(z)Y_n)^k R_n^{'}(z) u_j.
\end{equation*}
By Assumption \ref{assump:Xb}, 
almost surely for any $z \in \Gamma,$ there exists $\eta_z >0$ such that for all large $n,$
\begin{equation}\label{eq:norm-R_n'}
\Vert R_n^{'} (z) \Vert \leq \frac{1}{\eta_z}.
\end{equation}
Therefore, Proposition \eqref{prop:isotropic} implies that 
\begin{equation*}
v_i^* (R_n^{'}(z)Y_n)^k R_n^{'}(z) u_j \rightarrow 0 \; \text{in probability}
\end{equation*}
as $n$ goes to infinity. Thus, our result holds by applying Proposition \ref{prop:norm-resolvent} and the dominated convergence theorem.
\end{proof}

\begin{proposition}\label{prop:norm-difference}
Suppose that Assumption \ref{assump:Ab} holds and Assumptions \ref{assump:Xa}--\ref{assump:Xb} hold with $\Gamma \subseteq \Theta_{\rm out}$ compact. We have 
\begin{equation*}
\sup_{z \in \Gamma} \left \Vert Q_n R_n(z) P_n - Q_n R_n^{'}(z) P_n \right\Vert \rightarrow 0 \; \text{in probability}
\end{equation*}
as $n$ goes to infinity.
\end{proposition}

\begin{proof}
Firstly, for any $z \in \Gamma,$ let $C'>0$ such that 
\begin{equation*}
\Vert P_n \Vert \cdot \Vert Q_n \Vert \leq C'.
\end{equation*}
By Proposition \ref{prop:no-outilier}, there exists $\gamma_z >0$ such that almost surely for all large $n,$
\begin{equation*}
\Vert R_n (z) \Vert \leq \frac{1}{\gamma_z}.
\end{equation*}
Then by Proposition \ref{prop:norm-resolvent} and Equation \eqref{eq:norm-R_n'}, we have for any $k \geq 1$
\begin{equation*}
\Vert  Q_n (R_n^{'}(z) Y_n)^k R_n^{'}(z) P_n \Vert \leq \frac{C' C_{\ref{prop:norm-resolvent}}}{\eta_z} (1-\varepsilon_0)^k 
\end{equation*}
and 
\begin{equation*}
\Vert  Q_n (R_n^{'}(z) Y_n)^k R_n(z) P_n \Vert \leq \frac{C' C_{\ref{prop:norm-resolvent}}}{\gamma_z} (1-\varepsilon_0)^k.
\end{equation*}
For any $\varepsilon>0,$ choose $K$ such that 
\begin{equation*}
\frac{C' C_{\ref{prop:norm-resolvent}}}{\gamma_z} (1-\varepsilon_0)^K < \varepsilon/2 \; \text{and} \; \sum_{k \geq K} \frac{C' C_{\ref{prop:norm-resolvent}}}{\eta_z} (1-\varepsilon_0)^k < \varepsilon/2.
\end{equation*} 
Then, Equation \eqref{eq:induction-R_n} implies that
\begin{equation*}
\left\Vert Q_n R_n(z) P_n -Q_n R_n^{'}(z) P_n- \sum_{k\geq 1} (-1)^k Q_n (R_n^{'}(z) Y_n)^k R_n^{'}(z) P_n \right\Vert < \varepsilon.
\end{equation*}
Thus, we have (by letting $\varepsilon \to 0$)
\begin{equation}
Q_n R_n(z) P_n -Q_n R_n^{'}(z) P_n = \sum_{k\geq 1} (-1)^k Q_n (R_n^{'}(z) Y_n)^k R_n^{'}(z) P_n.
\end{equation}
Hence, due to Lemma \ref{lem:series-R_n}, we have for any $z \in \Gamma,$
\begin{equation}\label{eq:pointwise-conver}
 \left \Vert Q_n R_n(z) P_n - Q_n R_n^{'}(z) P_n \right\Vert \rightarrow 0 \; \text{in probability}
 \end{equation}
as $n$ goes to infinity. 

Finally, Proposition \ref{prop:norm-difference} follows by using a standard compactness argument; see \cite[Proposition 6.9]{Serban2021}. For any $\varepsilon>0,$ set 
$\xi_z = \min( \eta_z, \gamma_z)$ and $r_z = \min(\frac{\xi_z}{2}, \frac{\varepsilon\xi_z^2}{2C'}).$ Using the resolvent identity, if $(z, w) \in \Gamma^2$ such that $|z-w| \leq r_z,$ then almost surely for all large $n,$ we have
\begin{equation*}
\begin{split}
\left \Vert Q_n R_n(z) P_n - Q_n R_n(w) P_n \right\Vert & \leq \frac{2C'}{\xi_z^2} |z- w|\leq \varepsilon, \\
\left \Vert Q_n R_n'(z) P_n - Q_n R_n'(w) P_n \right\Vert & \leq \frac{2C'}{\xi_z^2} |z- w|\leq \varepsilon. 
\end{split}
\end{equation*}
Since $\Gamma$ is compact, there exist $z_1, z_2, \ldots, z_k \in \Gamma,$ such that $\Gamma \subseteq \cup_{i=1}^k B(z_i, r_{z_i}).$ So for any $z \in \Gamma,$ there exists $z_i,$ such that  almost surely for all large $n,$
\begin{equation*}
 \left \Vert Q_n R_n(z) P_n - Q_n R_n(z) P_n \right\Vert  \leq 2 \varepsilon +  \left \Vert Q_n R_n(z_i) P_n - Q_n R_n(z_i) P_n \right\Vert. 
\end{equation*}
Thus, for all large $n,$
\begin{equation*}
\begin{split}
 &\mathbb{P} \left( \sup_{z \in \Gamma} \left \Vert Q_n R_n(z) P_n - Q_n R_n(z) P_n \right\Vert  < 3 \varepsilon \right) \\
 &\quad \quad \quad \geq \sum_{i=1}^k  \mathbb{P} \left( \left \Vert Q_n R_n(z_i) P_n - Q_n 
 R_n(z_i) P_n \right\Vert  < \varepsilon \right),
 \end{split}
\end{equation*}
and the proposition follows form \eqref{eq:pointwise-conver}.
\end{proof}

\begin{proof}[Proof of Theorem \ref{thm:outlier}]
Recall that 
\begin{equation*}
f_n(z) = \det \left( \un_r - Q_n R_n(z) P_n \right)
\end{equation*}
and
\begin{equation*}
g_n(z) = \det \left( \un_r - Q_n R_n^{'}(z) P_n \right).
\end{equation*}
Note that
\begin{equation}\label{eq:determinant}
\det \left( A_n +Y_n-z  \right)= \det \left( A_n^{'}+ Y_n -z\right)\cdot f_n(z).
\end{equation}
According to Theorem \ref{thm:no-outlier}, almost surely for all large $n,$ the matrix $A_n^{'}+Y_n -z$ is invertible for any $z \in \Gamma.$ Hence, almost surely for all large $n,$ the eigenvalues of $A_n +Y_n$ in $\Gamma$ are exactly the zeros of the function $f_n.$

On the other hand, we have for all $z \in \Gamma,$ 
\begin{equation}
g_n(z) = \frac{\det \left( A_n-z \right)}{\det \left( A_n^{'}-z \right)}.
\end{equation}
Thus, Assumption \ref{assump:Xb} implies that the eigenvalues of $A_n$ in $\Gamma$ are the zeros of the function $g_n.$
Moreover, by Proposition \ref{prop:norm-difference}, uniformly on $\Gamma,$ $f_n(z)- g_n(z)$ converges to $0$ in probability as $n$ goes to infinity. Hence, the Rouch\'{e} theorem implies that  
$f_n$ and $g_n$ have the same number of zeros on $\Gamma,$ and our result holds.  
\end{proof}


\section{Lack of eigenvalues in the inner domain; proof of Theorem \ref{thm:outlier-inner}}

Recall that 
\begin{equation*}
f_n(z) = \det \left( \un_r - Q_n R_n(z) P_n \right).
\end{equation*}
Our goal is to show that, with probability tending to one, the function $f_n(z)$ has no zero for $z \in \Gamma \subseteq \Theta_{\rm in}.$ Hence, a simple condition would be 
$\Vert Q_n R_n(z) P_n \Vert <1$ for all $z \in \Gamma.$ 
Similar to the approach in \cite{Benaych-Georges-2016}, the idea is to consider $Y_n^{-1}$ instead of $Y_n$ by assuming the same hypothesis for $Y_n$. 

Recall that for any $\omega, z \in \mathbb{C}$,  $\mu_{\omega, z}$ is the distribution of $(a+T-z) (a+T-z)^*.$ Then we have the following counterpart of Lemma \ref{lem:mu_{omega, z}}:

\begin{lemma}\label{lem:mu_{omega, z}-1}
Let $\Gamma$ be a compact subset in $\Theta_{\rm in}.$  Then there exists $\rho' >1$ such that for any $\omega \in \mathbb{C}$ such that 
$|\omega|^{-1} \leq \rho'$ and any $z \in \Gamma,$ we have $0 \notin {\rm supp} (\mu_{\omega, z}).$
\end{lemma}

\begin{proof}
Let $\mu_1 =\tilde{\mu}_{|a-z|}$ and $\mu_2= \tilde\mu_{|T|},$ and denote $\mu_{2, \omega}: = \tilde{\mu}_{|\omega T|}.$
Then we have $0 \notin \supp (\mu_{\omega, z})$ if $0 \notin \supp (\mu_1\boxplus \mu_{2, \omega}).$
Hence, our results follows from Lemma \ref{lem:mu_{omega, z}-0}.
\end{proof}

\begin{proposition}\label{prop:spectral-radius-1}
Let $\Gamma$ be a compact subset in $\Theta_{\rm in}.$ Assume that Assumption \ref{assump:Ab} and Assumptions \ref{assump:Xa}-\ref{assump:Xb}-\ref{assump:Xd}  hold. There exists $0 < \varepsilon_0 <1$ such that almost surely for large $n$, we have, 
\begin{equation*}
\sup_{z \in \Gamma} r ((A_n^{'}-z) Y_n^{-1}) \leq 1- \varepsilon_0,
\end{equation*}
where $r (X_n)$ is the spectral radius of a matrix $X_n.$
\end{proposition}

\begin{proof}
We adapt the proof for Proposition \ref{prop:spectral-radius}.
Fix $\xi'>0$ such that $\xi'< \rho',$ where$\rho'$ is given in Lemma \ref{lem:mu_{omega, z}-1}. Let $\tilde{\Gamma} = \{(\omega, z) \in \mathbb{C}^2: \xi' \leq |\omega|^{-1} \leq \rho', z \in \Gamma\}.$  Hence, $0 \notin {\rm supp} (\mu_{\omega, z})$ for any 
$(\omega, z) \in \tilde{\Gamma}.$ By the proof of Proposition \ref{prop:no-outilier}, there exists $\gamma'_{\omega, z} >0$ such that almost surely for all large $n,$ there is no singular value of 
\begin{equation*}
\omega Y_n + A_n^{'} -z  
\end{equation*}
in $[0, \gamma'_{\omega, z}].$ 

On the other hand, by the resolvent identity, we have
\begin{equation}
\omega Y_n + A_n^{'} -z  = \omega Y_n \left[ (A_n^{'}-z)^{-1} + (\omega Y_n)^{-1} \right] (A_n^{'}-z).
\end{equation}
It follows that
\begin{equation*}
\begin{split}
s_n ((A_n^{'}-z)^{-1} + (\omega Y_n)^{-1} ) & \geq s_n ((\omega Y_n)^{-1}) \cdot s_n \left( \omega Y_n + A_n^{'}-z \right) \cdot  s_n ((A_n^{'}-z)^{-1})\\
& \geq \xi' \gamma'_{\omega, z} \eta'_z s_n (Y_n^{-1})
\end{split}
\end{equation*}
for all large $n.$ Moreover, it follows from Assumption \ref{assump:Xd} that $\|(Y_n)^{-1}\| \leq C_{\ref{assump:Xd}}.$ Then, using a compactness argument, we can prove that there exists $\eta >0$ such that almost surely for all large $n,$ there is no singular value of 
$(A_n^{'} - z)^{-1} + (\omega Y_n)^{-1}$ in $[0, \eta]$ for any $(\omega, z) \in \tilde{\Gamma}.$

Assume that $\lambda \neq 0$ is an eigenvalue of $(A_n^{'}-z) Y_n^{-1}.$ Then, $0$ must be a singular value of $-\lambda^{-1}Y_n^{-1} + (A_n^{'}-z)^{-1}.$ Therefore, we must have $1/|\lambda| >\rho'>1,$ which completes the proof.
\end{proof}

\begin{proposition}\label{prop:norm-resolvent-1}
Let $\Gamma$ be a compact subset in $\Theta_{\rm in}.$ Assume that Assumption \ref{assump:Ab} and Assumptions \ref{assump:Xa}-\ref{assump:Xb}-\ref{assump:Xd}  hold. There exists $0 < \varepsilon_0 <1$ and $C_{\ref{prop:norm-resolvent}} >0$ such that almost surely for large $n$, for any $k \geq 1$, we have 
\begin{equation*}
\sup_{z \in \Gamma} \Vert ((A_n^{'}-z) Y_n^{-1})^k \Vert \leq C_{\ref{prop:norm-resolvent-1}} (1- \varepsilon_0)^k.
\end{equation*}
\end{proposition}

\begin{proof}
We adapt the proof for Proposition \ref{prop:norm-resolvent}. For $z \in \Gamma,$ set $T_n^{'} (z)= (A_n^{'}-z) Y_n^{-1}.$ 
Choose $0< \varepsilon < \min\{\varepsilon_0, 1-1/\rho'\},$ where $\varepsilon_0$ and $\rho'$ are given in Proposition \ref{prop:spectral-radius-1} and Lemma \ref{lem:mu_{omega, z}-1}, respectively. 
By Proposition \ref{prop:spectral-radius-1} and Dunford-Riesz calculus, we have almost surely for all large $n,$
\begin{equation*}
(T_n^{'}(z))^k = \frac{1}{2\pi i} \int_{|\omega| =1-\varepsilon} \omega^k (\omega - T_n^{'}(z))^{-1} {\rm d} \omega, \;\; k \geq 0.
\end{equation*}
Hence,
\begin{equation*}
\Vert (T_n^{'}(z))^k  \Vert \leq  (1-\varepsilon)^{k+1} \cdot \sup_{|\omega| =1-\varepsilon} \Vert (\omega - T_n^{'}(z))^{-1} \Vert, \;\; k \geq 0.
\end{equation*}

For any $\omega$ such that $|\omega| = 1- \varepsilon,$ we have 
\begin{equation*}
\omega -T_n^{'} (z) = \omega (A_n^{'}-z) \left( -\frac{1}{\omega} Y_n^{-1} + (A_n^{'}-z)^{-1} \right).
\end{equation*}
Since $1/|\omega| < \rho',$ there exists $\eta >0$ such that there is no singular value of $-\frac{1}{\omega} Y_n^{-1} + (A_n^{'}-z)^{-1}$  in $[0, \eta];$ see the proof of Proposition \ref{prop:spectral-radius-1}. 
It follows that $\Vert ( -\frac{1}{\omega} Y_n^{-1} + (A_n^{'}-z)^{-1}  )^{-1}\Vert \leq \frac{1}{\eta}.$ Therefore, there exists $C>0$, such that for any $z \in \Gamma,$
\begin{equation*}
 \Vert (\omega - T_n^{'}(z))^{-1} \Vert  \leq \frac{1}{|\omega|} \left\Vert \left( -\frac{1}{\omega} Y_n^{-1} + (A_n^{'}-z)^{-1}  \right)^{-1} \right\Vert  \cdot \Vert (A_n^{'} -z)^{-1}\Vert \leq C,
\end{equation*}
which completes our proof.
\end{proof}

\begin{proposition}\label{prop:norm-convergence-inner}
Let $\Gamma$ be a compact subset in $\Theta_{\rm in}.$ Assume that Assumption \ref{assump:Ab} and Assumptions \ref{assump:Xa}-\ref{assump:Xb}-\ref{assump:Xd}  hold. For any $z \in \Gamma,$ we have 
\begin{equation*}
\Vert Q_n R_n(z) P_n \Vert \rightarrow 0 \;\; \text{in probability}
\end{equation*}
as $n$ goes to infinity, 
\end{proposition}

\begin{proof}
Note that $Y_n$ is invertible. For any $z \in \Gamma,$ 
\begin{equation*}
\begin{split}
Q_n R_n(z) P_n &  =Q_n Y_n^{-1} \left( \un_n- (z- A_n^{'}) Y_n^{-1} \right)^{-1} P_n\\
& = Q_n \sum_{k=0}^\infty Y_n^{-1}  ((z- A_n^{'}) Y_n^{-1})^k P_n.
\end{split}
\end{equation*}
By repeating the proof of Proposition \ref{prop:isotropic},  for any 
sequence uniformly bounded matrices $(B^{(1)}, \ldots, B^{(k-1)})\in M_n(\mathbb{C})^{k-1}$ and unit vectors $c, b$ in $\mathbb{C}^n$, we have that
\begin{equation*}
c^* P \left(B^{(1)}, \ldots, B^{(k-1)}, Y^{-1}_n \right) b  \rightarrow 0 \;\; \text{in probability}
\end{equation*}
as $n \rightarrow \infty.$

Note that $\sup_n \Vert z- A_n^{'} \Vert \leq |z| + C_{\ref{assump:Xb}}$, for any $z \in \Gamma.$ It implies that, for all $k \geq 0,$ 
\begin{equation}\label{eq:isotropy}
c^* Y_n^{-1}  ((z- A_n^{'}) Y_n^{-1} )^k b  \rightarrow 0 \; \; \text{in probability}
\end{equation}
as $n \rightarrow \infty.$
Note that the singular value decomposition of $A_n^{''}$ gives that for any $i, j =1, \ldots, r,$
\begin{equation*}
(Q_n Y_n^{-1}  ((z- A_n^{'}) Y_n^{-1} )^k P_n)_{i, j} = s_i v_i^*Y_n^{-1}  ((z- A_n^{'}) Y_n^{-1} )^k u_j,
\end{equation*}
where $v_i$ and $u_j$ are unit column vectors and $s_i$ is a singular value of $A_n^{''}.$
Therefore, Equation \eqref{eq:isotropy} implies that 
\begin{equation*}
v_i^* Y_n^{-1}  ((z- A_n^{'}) Y_n^{-1} )^k u_j \rightarrow 0 \; \text{in probability}
\end{equation*}
as $n$ goes to infinity. And  by using the dominated convergence theorem, Proposition \ref{prop:norm-resolvent-1} yields that 
\begin{equation*}
\left \Vert Q_n \sum_{k=0}^\infty Y_n^{-1}  ((z- A_n^{'}) Y_n^{-1})^k P_n \right \Vert \rightarrow 0 \; \text{in probability}
\end{equation*}
as $n$ goes to infinity. 
\end{proof}

\begin{proof}[Proof of Theorem \ref{thm:outlier-inner}]
We first recall that
\[
 f_n(z) = \frac{\det \left( A_n+Y_n-z \right)}{\det \left( A_n^{'}+Y_n-z \right)}.
\]
 Theorem \ref{thm:no-outlier} induces that almost surely for all large $n,$ $\det \left( A_n^{'}+Y_n-z \right)\neq 0$ for  $z\in\Gamma\subset \Theta_{\rm in}$.
By Proposition \ref{prop:norm-convergence-inner}, with probability tending to one, the function $f_n(z)=\det \left( \un_r - Q_n R_n(z) P_n \right)$ has no zero for $z \in \Gamma \subseteq \Theta_{\rm in}.$ 
Therefore, in probability for all large $n,$ the matrix $A_n +Y_n$ has no eigenvalues in $\Gamma$.
This finishes our proof. 
\end{proof}


\bigskip

{\bf Acknowledgment.} C-W. Ho is partially supported by the NSTC grant 111-2115-M-001-011-MY3. Z. Yin is supported by Provincial Natural Science Foundation of Hunan (grant number 2024JJ1010) and NSFC 12031004. P. Zhong is supported in part by NSF grant LEAPS-MPS-2316836 and NSF CAREER Award DMS-2339565.

\bibliographystyle{acm}
\bibliography{Outlier-deform-single-ring}

\begin{thebibliography}{10}

\bibitem{AltK2024Brown}
{\sc Alt, J., and Kr{\"u}ger, T.}
\newblock Brown measures of deformed $l^\infty$-valued circular elements.
\newblock {\em arXiv:2409.15405\/} (2024).

\bibitem{BaoES2019singlering}
{\sc Bao, Z., Erdos, L., and Schnelli, K.}
\newblock Local single ring theorem on optimal scale.
\newblock {\em Ann. Probab. 47}, 3 (2019), 1270--1334.

\bibitem{BelinschiBercoviciHo2022}
{\sc Belinschi, S., Bercovici, H., and Ho, C.-W.}
\newblock On the convergence of {D}enjoy-{W}olff points.
\newblock {\em arXiv:2203.16728\/} (2022).

\bibitem{Belinschi2008}
{\sc Belinschi, S.~T.}
\newblock The {L}ebesgue decomposition of the free additive convolution of two
  probability distributions.
\newblock {\em Probab. Theory Related Fields 142}, 1-2 (2008), 125--150.

\bibitem{BBC2021imrn}
{\sc Belinschi, S.~T., Bercovici, H., and Capitaine, M.}
\newblock On the outlying eigenvalues of a polynomial in large independent
  random matrices.
\newblock {\em Int. Math. Res. Not. IMRN}, 4 (2021), 2588--2641.

\bibitem{BBCF2017aop}
{\sc Belinschi, S.~T., Bercovici, H., Capitaine, M., and F\'evrier, M.}
\newblock Outliers in the spectrum of large deformed unitarily invariant
  models.
\newblock {\em Ann. Probab. 45}, 6A (2017), 3571--3625.

\bibitem{Serban2021}
{\sc Belinschi, S.~T., Bordenave, C., Capitaine, M., and C{\'e}bron, G.}
\newblock {Outlier eigenvalues for non-Hermitian polynomials in independent
  i.i.d. matrices and deterministic matrices}.
\newblock {\em Electron. J. Probab. 26\/} (2021), 1 -- 37.

\bibitem{Serban2017JFA}
{\sc Belinschi, S.~T., and Capitaine, M.}
\newblock Spectral properties of polynomials in independent {W}igner and
  deterministic matrices.
\newblock {\em J. Funct. Anal. 273}, 12 (2017), 3901--3963.

\bibitem{BelinschiYinZhong2021Brown}
{\sc Belinschi, S.~T., Yin, Z., and Zhong, P.}
\newblock The {B}rown measure of a sum of two free random variables, one of
  which is triangular elliptic.
\newblock {\em Adv. Math. 441\/} (2024), 109562.

\bibitem{BenaychGeorges2017}
{\sc Benaych-Georges, F.}
\newblock Local single ring theorem.
\newblock {\em Ann. Probab. 45}, 6A (2017), 3850--3885.

\bibitem{Benaych-Georges-2016}
{\sc Benaych-Georges, F., and Rochet, J.}
\newblock Outliers in the single ring theorem.
\newblock {\em Probab. Theory Related Fields 165}, 1 (2016), 313--363.

\bibitem{BercoviciZhong2022}
{\sc Bercovici, H., and Zhong, P.}
\newblock The {B}rown measure of a sum of two free nonselfadjoint random
  variables, one of which is {R}-diagonal.
\newblock {\em arXiv:2209.12379v2\/} (2025).

\bibitem{Biane1997}
{\sc Biane, P.}
\newblock On the free convolution with a semi-circular distribution.
\newblock {\em Indiana Univ. Math. J. 46}, 3 (1997), 705--718.

\bibitem{BordenaveC_cpam2016_outlier}
{\sc Bordenave, C., and Capitaine, M.}
\newblock Outlier eigenvalues for deformed i.i.d. random matrices.
\newblock {\em Comm. Pure Appl. Math. 69}, 11 (2016), 2131--2194.

\bibitem{BH2022}
{\sc Brailovskaya, T., and van Handel, R.}
\newblock Universality and sharp matrix concentration inequalities.
\newblock {\em Geom. Funct. Anal. 34\/} (2024), 1734--1838.

\bibitem{Brown1986}
{\sc Brown, L.~G.}
\newblock Lidskii's theorem in the type {II} case.
\newblock In {\em Geometric methods in operator algebras ({K}yoto, 1983)},
  vol.~123 of {\em Pitman Res. Notes Math. Ser.} Longman Sci. Tech., Harlow,
  1986, pp.~1--35.

\bibitem{CapitaineD2017survey}
{\sc Capitaine, M., and Donati-Martin, C.}
\newblock Spectrum of deformed random matrices and free probability.
\newblock In {\em Advanced topics in random matrices}, vol.~53 of {\em Panor.
  Synth\`eses}. Soc. Math. France, Paris, 2017, pp.~151--190.

\bibitem{CDFF2011}
{\sc Capitaine, M., Donati-Martin, C., F{\'e}ral, D., and F{\'e}vrier, M.}
\newblock {Free convolution with a semicircular distribution and eigenvalues of
  spiked deformations of Wigner matrices}.
\newblock {\em Electron. J. Probab. 16\/} (2011), 1750 -- 1792.

\bibitem{Collins-IMRN03}
{\sc Collins, B.}
\newblock {Moments and cumulants of polynomial random variables on unitary
  groups, the Itzykson-Zuber integral, and free probability}.
\newblock {\em Int. Math. Res. Not. IMRN 2003}, 17 (01 2003), 953--982.

\bibitem{CM2014}
{\sc Collins, B., and Male, C.}
\newblock The strong asymptotic freeness of {H}aar and deterministic matrices.
\newblock {\em Annales scientifiques de l'ENS 47}, 4 (2014), 147--163.

\bibitem{Collins-CMP06}
{\sc Collins, B., and {\'S}niady, P.}
\newblock Integration with respect to the {H}aar measure on unitary, orthogonal
  and symplectic group.
\newblock {\em Comm. Math. Phys. 264}, 3 (2006), 773--795.

\bibitem{FeinbergSZ2001-single-ring-2}
{\sc Feinberg, J., Scalettar, R., and Zee, A.}
\newblock ``{S}ingle ring theorem'' and the disk-annulus phase transition.
\newblock {\em J. Math. Phys. 42}, 12 (2001), 5718--5740.

\bibitem{GuionnetKZ-single-ring2011}
{\sc Guionnet, A., Krishnapur, M., and Zeitouni, O.}
\newblock The single ring theorem.
\newblock {\em Ann. of Math. (2) 174}, 2 (2011), 1189--1217.

\bibitem{HaagerupLarsen2000}
{\sc Haagerup, U., and Larsen, F.}
\newblock Brown's spectral distribution measure for {R}-diagonal elements in
  finite von {N}eumann algebras.
\newblock {\em J. Funct. Anal. 176}, 2 (2000), 331 -- 367.

\bibitem{HaagerupSchultz2007}
{\sc Haagerup, U., and Schultz, H.}
\newblock Brown measures of unbounded operators affiliated with a finite von
  {N}eumann algebra.
\newblock {\em Math. Scand. 100}, 2 (2007), 209--263.

\bibitem{HoHall2020Brown}
{\sc Hall, B.~C., and Ho, C.-W.}
\newblock The {B}rown measure of the sum of a self-adjoint element and an
  imaginary multiple of a semicircular element.
\newblock {\em Lett Math Phys 112}, 2 (2022), 19.

\bibitem{Ho2020Brown}
{\sc Ho, C.-W.}
\newblock The {B}rown measure of the sum of a self-adjoint element and an
  elliptic element.
\newblock {\em Electron. J. Probab. 27\/} (2022), 1--32.

\bibitem{HoZhong2020Brown}
{\sc Ho, C.-W., and Zhong, P.}
\newblock Brown measures of free circular and multiplicative brownian motions
  with self-adjoint and unitary initial conditions.
\newblock {\em J. Eur. Math. Soc. 25\/} (2023), 2163--2227.

\bibitem{Ho-Zhong23}
{\sc Ho, C.-W., and Zhong, P.}
\newblock Deformed single ring theorems.
\newblock {\em J. Funct. Anal. 288\/} (2025), 110797.

\bibitem{KnowlesYin2024aop}
{\sc Knowles, A., and Yin, J.}
\newblock The outliers of a deformed {W}igner matrix.
\newblock {\em Ann. Probab. 42}, 5 (2014), 1980--2031.

\bibitem{MingoSpeicherBook}
{\sc Mingo, J.~A., and Speicher, R.}
\newblock {\em Free probability and random matrices}, vol.~35 of {\em Fields
  Institute Monographs}.
\newblock Springer, New York; Fields Institute for Research in Mathematical
  Sciences, Toronto, ON, 2017.

\bibitem{NicaSpeicher-Rdiag}
{\sc Nica, A., and Speicher, R.}
\newblock {$R$}-diagonal pairs---a common approach to {H}aar unitaries and
  circular elements.
\newblock In {\em Free probability theory ({W}aterloo, {ON}, 1995)}, vol.~12 of
  {\em Fields Inst. Commun.} Amer. Math. Soc., Providence, RI, 1997,
  pp.~149--188.

\bibitem{SpeicherNicaBook}
{\sc Nica, A., and Speicher, R.}
\newblock {\em Lectures on the Combinatorics of Free Probability}.
\newblock London Mathematical Society Lecture Note Series. Cambridge University
  Press, 2006.

\bibitem{PRSoshnikov2013AIHP}
{\sc Pizzo, A., Renfrew, D., and Soshnikov, A.}
\newblock On finite rank deformations of {W}igner matrices.
\newblock {\em Ann. Inst. Henri Poincar\'e{} Probab. Stat. 49}, 1 (2013),
  64--94.

\bibitem{RV2014jams}
{\sc Rudelson, M., and Vershynin, R.}
\newblock Invertibility of random matrices: unitary and orthogonal
  perturbations.
\newblock {\em J. Amer. Math. Soc. 27}, 2 (2014), 293--338.

\bibitem{Tao-2013}
{\sc Tao, T.}
\newblock Outliers in the spectrum of iid matrices with bounded rank
  perturbations.
\newblock {\em Probab. Theory Related Fields 155}, 1 (2013), 231--263.

\bibitem{Voiculescu1991}
{\sc Voiculescu, D.}
\newblock Limit laws for random matrices and free products.
\newblock {\em Invent. Math. 104}, 1 (1991), 201--220.

\bibitem{Weingarten-78}
{\sc Weingarten, D.}
\newblock {Asymptotic behavior of group integrals in the limit of infinite
  rank}.
\newblock {\em J. Math. Phys. 19}, 5 (05 1978), 999--1001.

\bibitem{Zhong2021Brown_ctgt}
{\sc Zhong, P.}
\newblock Brown measure of the sum of an elliptic operator and a free random
  variable in a finite von neumann algebra.
\newblock {\em arXiv:2108.09844v4\/} (2021).

\end{thebibliography}

\section{Appendix: global left inverses of subordination functions}

In this appendix, we will study the results in Subsection \ref{subsec:free-convolution} for a special case when $\mu_2$ is a $\boxplus$-infinite divisible measure. The $\boxplus$-infinite divisibility allows us to express the subordination functions via 
an integral formula. Hence, we are able to describe the support of $\mu_1 \boxplus \mu_2$ explicitly. 

\begin{definition}[\cite{SpeicherNicaBook}]
Let $\mu$ be a compactly supported probability measure on $\mathbb{R}.$ $\mu$ is called $\boxplus$-infinite divisible if, for each positive integer $k,$ there exists a probability measure $\mu_k,$ such that
\begin{equation*}
\mu = (\mu_k)^{\boxplus k}. 
\end{equation*}
\end{definition}

Let $\mu_2$ be a compactly supported $\boxplus$-infinite divisible measure, its Voiculescu transform is given by
\begin{equation}
\varphi_{2}(z) = \int_{\mathbb{R}} \frac{{\rm d} \sigma (\xi)}{z-\xi}, \; z \in \mathbb{C}^+,
\end{equation}
where $\sigma$ is a positive finite symmetric measure on $\mathbb{R}.$

\begin{remark}
In general, we have 
\begin{equation*}
\varphi_2 (z) = \eta + \int_{\mathbb{R}} \frac{1+\xi z}{z-\xi} {\rm d}\sigma (\xi).
\end{equation*} 
However, $\eta = 0$ because $\mu_2$ is symmetric, and we can write the integral as the Cauchy transform instead because we always deal with measures with compact support.
\end{remark}

Let $\mu_1$ and $(\mu_1^{(n)})_{n \geq 1}$ be symmetric probability measures on $\mathbb{R}.$ Since $\mu_2$ is $\boxplus$-infinite divisible, so we have the following explicit formulas for $\varphi$ and $\varphi_n.$
\begin{equation}\label{eq.Frho.inverse}
\varphi(z) = z+ \int_{\mathbb{R}} \frac{{\rm d} \sigma (\xi)}{F_1 (z)-\xi}, \;\; \varphi_n (z) = z+\int_{\mathbb{R}} \frac{{\rm d} \sigma (\xi)}{F_{1, n}(z)-\xi}
\end{equation}
for $z \in \mathbb{C}^+.$ 
And
\begin{equation}
\omega_1(z) = z -  \int_{\mathbb{R}} \frac{{\rm d} \sigma (\xi)}{F_1 (\omega_1 (z))-\xi}, \;\; \omega_{1, n} (z) = z -  \int_{\mathbb{R}} \frac{{\rm d} \sigma (\xi)}{F_{1, n} (\omega_{1, n} (z))-\xi}
\end{equation}
for $z \in \mathbb{C}^+.$ Moreover, by using Vioculescu's transform is additive with respect to the free convolution, we have 
\begin{equation}\label{eq.subordination}
F_{\mu_1\boxplus \mu_2}(\varphi (z))=F_1 (z), \;\; F_{\mu_1^{(n)} \boxplus \mu_2}(\varphi_n(z))=F_{1, n}(z)
\end{equation}
for $z \in \mathbb{C}^+.$ Recall that we always have the following assumptions:

\begin{assumption}\label{assumptions}
Let $\mu_1$, $\mu_2$ and $\mu_1^{(n)}$ be compactly supported symmetric measures on $\mathbb{R}$. We assume that
\begin{enumerate}[(a)]
        \item $\mu_1^{(n)} \to \mu_1$ weakly as $n \to \infty.$
        \item There exists $\delta>0$ such that $\supp\mu_1^{(n)} \cap[-\delta,\delta] = \emptyset$ for all large $n.$
        \item $\supp \mu_2 \subset [-r, r]$. Since $F_1$ has an analytic continuation to $(0,\delta]$ and
        \[\lim_{x \to 0^+} F_1(x) = -\infty,\]
        we assume $\delta>0$ is chosen small enough that $F_1(x) <-r-1$ for all $0<x<2\delta$. By symmetry,  $F_1 (x)>r+1$ for all $-2\delta<x<0$.
        \item $0\not\in\supp(\mu_1\boxplus\mu_2)$.
    \end{enumerate}
\end{assumption}
Observe that for each $\xi\in[-r,r]$, $z \to F_1(z)-\xi$ is the $F$-transform of some measure. Thus, for each $\xi\in[-r,r]$, let $\lambda_{\xi,n}$ and $\lambda_\xi$ be probability measures such that
  \begin{equation}
  F_{\lambda_{\xi,n}} = F_{1, n}(z)-\xi, \;\;  F_{\lambda_\xi}(z) = F_1(z)-\xi.
  \end{equation}

\begin{lemma}
    \label{lem.lambda.support}
    Let $\delta>0$ be chosen in Assumption \ref{assumptions}. Then 
    \[\supp\lambda_{\xi,n}\cap[-\delta,\delta] = \supp\lambda_
    \xi\cap[-\delta,\delta] =\emptyset\]
    for all large $n$. Furthermore, $\lambda_{-\xi}(B) = \lambda_\xi(-B)$ for any Borel set $B$.
\end{lemma}

\begin{proof}
   The function $x \to F_{1, n}(x)$ is strictly increasing on $(0,\delta]$ and strictly decreasing on $[-\delta,0)$. Since $\mu_1^{(n)} \to\mu$ weakly, $F_{1, n}(\pm\delta)\to F_1(\pm \delta)$ as $n \to \infty$. It follows that, for all large $n$, 
    \[\vert F_{1, n}(x)\vert > r+1\] 
    for all $x\in [-\delta,0)\cup (0,\delta]$. Thus, for all large $n$, 
    \[\lim_{y\to 0^+}-\mathrm{Im} G_{\lambda_{\xi,n}}(x+iy)=\lim_{y\to 0^+}-\mathrm{Im} G_{\lambda_{\xi}}(x+iy) = 0\]
    whenever $x\in[-\delta,0)\cup (0,\delta]$. Meanwhile, since $F_1(iy)$ is purely imaginary with posiitve imaginary part,
    \[-\mathrm{Im}G_{\lambda_\xi,n}(iy) = \frac{\vert F_{1, n}(iy)\vert}{\vert F_{1, n}(iy)\vert^2+\xi^2}\leq \frac{1}{\vert F_{1,n}(iy)\vert}\]
    and 
    \[-\mathrm{Im}G_{\lambda_\xi}(iy) = \frac{\vert F_1 (iy)\vert}{\vert F_1 (iy)\vert^2+\xi^2}\leq \frac{1}{\vert F_1(iy)\vert}.\]
    We have
    \[\lim_{y\to 0^+}-\mathrm{Im} G_{\lambda_{\xi,n}}(iy)=\lim_{y\to 0^+}-\mathrm{Im} G_{\lambda_{\xi}}(iy) = 0.\]
    By the Stieltjes inversion formula, $\supp\lambda_{\xi,n}\cap[-\delta,\delta] = \supp\lambda_
    \xi\cap[-\delta,\delta] =\emptyset$.

    The last assertion follows from the formula
    \begin{align*}
        \mathrm{Im}G_{\lambda_\xi}(x+iy) &= -\frac{\mathrm{Im} F_1 (-x+iy)}{(-\mathrm{Re}F_1 (-x+iy)-\xi)^2+\mathrm{Im}F_1 (-x+iy)}\\
        &= \mathrm{Im}G_{\lambda_\xi}(-x+iy)
    \end{align*}
because $F_1$ is symmetric about the imaginary axis.
\end{proof}

Now, we can compute $\omega_1 (\C^+\cup\R)$ and $\omega_{1, n}(\C^+\cup\R)$. 
\begin{proposition}
    \label{prop.bdry.function}
    There exist continuous functions $\nu_n, \nu: \R\to[0,\infty)$ such that
\begin{align*}
    &\omega_{1, n} (\C^+\cup\R) =\{x+iy: y\geq f_n(x)\} \\
    &\omega_1 (\C^+\cup\R) = \{x+iy: y\geq f(x)\}.
\end{align*}
    In fact, $\nu_n$ and $\nu$ have explicit formulas
    \begin{align*}
        &\nu_n(x) = \inf\left\{y>0: \int\int\frac{{\rm d}\lambda_{\xi,n}(t)}{(x-t)^2+y^2} {\rm d} \sigma (\xi)<1\right\}\\
        &\nu(x) = \inf\left\{y>0: \int\int\frac{{\rm d} \lambda_\xi(t)}{(x-t)^2+y^2}{\rm d} \sigma (\xi)<1\right\}.
    \end{align*}
\end{proposition}

\begin{proof}
    We can write 
   \begin{equation*}
   \varphi (z) = z + \int G_{\lambda_\xi}(z) {\rm d} \sigma (\xi)
   \end{equation*} 
  and compute
    \[\mathrm{Im} \varphi (x+iy) = y\left[1-\int\int\frac{{\rm d} \lambda_\xi(t)}{(x-t)^2+y^2} {\rm d} \sigma(\xi)\right].\]
    Similar to Biane's argument in \cite{Biane1997} (the second factor is increasing in $y$), the formula of $\nu$ is as desired. The formula of $\nu_n$ can be derived in a similar manner.
\end{proof}

Now, since we know from Lemma \ref{lem.lambda.support} the supports of $\lambda_{\xi, n}$ and $\lambda_\xi$ do not contain $[-\delta,\delta]$. We can study $\nu_n$ and $\nu$ through the double integrals.

\begin{lemma}
    \label{lem.convex.monotone}
    The functions
    \[d_n(x)=\int\int\frac{{\rm d} \lambda_{\xi,n}(t)}{(x-t)^2} {\rm d} \sigma(\xi) \; \text{and} \;  d(x)= \int\int\frac{ {\rm d} \lambda_{\xi}(t)}{(x-t)^2} {\rm d} \sigma(\xi)\]
    are convex for $x\in[-\delta,\delta]$. They are symmetric about $x=0$; in particular, they are strictly increasing for $0<x<\delta$ and strictly decreasing for $-\delta<x<0$.
\end{lemma}

\begin{proof}
    We only prove for the function $d$; the proof for $d_n$ is similar. Denote by 
    \[ \tilde\lambda_\xi = \begin{cases} \lambda_\xi+\lambda_{-\xi},& \xi \neq 0; \\ 
    \lambda_\xi,& \xi=0.\end{cases}\]
It is clear that $\tilde\lambda_\xi$ is symmetric. We can differentiate the function for $x\in[-\delta,\delta]$ because $\supp \lambda_\xi\cap[-\delta,\delta]=\emptyset$. Since the measure $\sigma$ is symmetric and $\lambda_{-\xi}(B) = \lambda_{\xi}(-B)$,
    \begin{align*}
        \frac{{\rm d}}{{\rm d} x}\int_\R\int_\R\frac{{\rm d} \lambda_{\xi}(t)}{(x-t)^2} {\rm d} \sigma(\xi) &=-2 \int_{[0,\infty)}\int_\R\frac{ {\rm d} \tilde\lambda_{\xi}(t)}{(x-t)^3}{\rm d} \sigma(\xi)\\
        &=-4x\int_{[0,\infty)}\int_{(\delta,\infty)}\frac{x^2+3t^2}{(x^2-t^2)^3}{\rm d} \tilde\lambda_{\xi}(t) {\rm d} \sigma(\xi)
    \end{align*}
    is positive for $0<x<\delta$ and negative for $-\delta<x<0$; we have again used the fact $\supp \lambda_\xi\cap[-\delta,\delta]=\emptyset$.

    The function $d(x)$ is convex because the second derivative is positive for all $x\in[-\delta,\delta]$. Moreover, by a similar computation as in the preceding paragraph, the function
    \[ d(x) = \int_{[0,\infty)}\int_{(\delta,\infty)}\frac{2(x^2+t^2)}{(x^2-t^2)^2} {\rm d} \lambda_{\xi}(t) {\rm d} \sigma(\xi)\] 
    is clearly symmetric about $x=0$.
\end{proof}

\begin{proposition}
    \label{prop.hn.sub1}
    Under Assumption \ref{assumptions} $(d)$, there exists $0<\gamma<\delta$ such that $d(x) < 1$ for $x \in (0, \delta].$ For all $n$ large enough, we also have $d_n(x) < 1$ for $x \in (0, \delta].$
\end{proposition}

\begin{proof}
  Recall that the function $\nu$ describes the boundary of $\omega_1$ as in Proposition \ref{prop.bdry.function}. Hence, the subordination relation \eqref{eq.subordination} guarantees that we can find a $\gamma$ such that $0<2\gamma<\delta$ and $\nu(x)=0$ for all $0<x<2\gamma$. Proposition \ref{prop.bdry.function} tells us the function $d$ in Lemma \ref{lem.convex.monotone} satisfies $d(2\delta)\leq 1$. Moreover, Lemma \ref{lem.convex.monotone}  shows that $h$ is strictly increasing in $(0,\delta)$, so that $d(x)<1$ for all $0<x<2\delta$. 

    Now, we compute, for all $0<x<2\gamma$
    \begin{align*}
        d(x)&=\int\int\frac{{\rm d} \lambda_\xi(t)}{(x-t)^2} {\rm d} \sigma(\xi) = -\frac{{\rm d}}{{\rm d} x}\int G_{\lambda_\xi}(x){\rm d} \sigma(\xi)\\
        &=-\frac{{\rm d}}{{\rm d}x}\int\frac{1}{F_1(x)-\xi} {\rm d} \sigma(\xi).
    \end{align*}
    By the preceding paragraph, $d(\gamma)<1$. Since $d$ is analytic in a neighborhood of $\gamma$, we must have (by applying Montel's theorem, for example) $d_n(\gamma)<1$ for all large $n$. Since $h_n$ is strictly increasing by Lemma \ref{lem.convex.monotone}, $d_n(x)<1$ for all $0<x\leq\gamma$.
\end{proof}

\begin{theorem}
 Suppose Assumption \ref{assumptions} holds and $\mu_2$ is $\boxplus$-infinite divisible. Then there exists $\varepsilon>0,$ such that $[-\varepsilon,\varepsilon]\cap\supp \mu_1^{(n)}\boxplus \mu_2 = \emptyset$ for all large $n.$
\end{theorem}

\begin{proof}
By Proposition \ref{prop.hn.sub1}, $d_n(x)<1$ and $d(x)<1$ for all $0<x\leq \gamma$; in particular, $\nu_n(x) = \nu(x)=0$ for all $0<x\leq \gamma$. For all $0<x<\delta$, $\varphi$ and $\varphi_n$ is differentiable because $\supp\lambda_{\xi,n}\cap[-\delta,\delta] = \supp\lambda_
    \xi\cap[-\delta,\delta] =\emptyset$. The derivative of $\varphi$ is given by 
    \[ \varphi'(x) = 1+\int G_{\lambda_\xi}'(x) {\rm d} \sigma(\xi) = 1-d(x).\]
    Similarly, $\varphi_n'(x) = 1-d_n(x)$. Thus, the property of $\gamma$ in Proposition \ref{prop.hn.sub1} shows that $\varphi'(x)>0$ and $\varphi_n'(x)>0$ for all $0<x\leq\gamma$. That is, $\varphi$ and $\varphi_n$ are strictly increasing on $[0,\gamma]$. We also have, in particular, $\varphi(\gamma)>0$.

Now, since $\varphi_n (\gamma)\to \varphi(\gamma)$ as $n\to\infty$, the fact that $\varphi(\gamma)>0$ shows that there exists $ \varepsilon \in (0, \delta)$ such that $\varphi_n(\gamma)>\varepsilon$ for all large $n$. Since $\sigma$ is symmetric, we have
    \[[-\varepsilon,\varepsilon]\subset \varphi_n ([-\gamma,\gamma])\]
    for all large $n$. Thus, the conclusion of the theorem now follows from the subordination relation \eqref{eq.subordination}.
\end{proof}

\end{document}